\documentclass[final,a4paper,11pt]{article}
\usepackage{amsmath,amssymb}
\usepackage{amsthm}
\usepackage{cite}
\usepackage[paper=a4paper, left=2.5cm, right=2.5cm, top=2.5cm, bottom=2.5cm]{geometry}
\usepackage{color}
\usepackage{hyperref}
\hypersetup{
    colorlinks=true,
    linkcolor=blue,     citecolor=red,     urlcolor=blue  }
\usepackage[colorinlistoftodos]{todonotes}
\usepackage{mathptmx}   %Times New Roman
\usepackage[weather,alpine,misc,geometry]{ifsym}
\usepackage{amsfonts}
\usepackage{graphicx}%
\usepackage[capitalise,nameinlink]{cleveref}
\crefformat{equation}{(#2#1#3)}
\crefmultiformat{equation}{(#2#1#3)}%
{ and~(#2#1#3)}{, (#2#1#3)}{ and~(#2#1#3)}
\crefrangeformat{equation}{(#3#1#4)--(#5#2#6)}
\crefrangemultiformat{equation}{(#3#1#4)--(#5#2#6)}%
{ and~(#3#1#4)--(#5#2#6)}{, (#3#1#4)--(#5#2#6)}{ and~(#3#1#4)--(#5#2#6)}
\usepackage[right]{showlabels}
\newcommand{\al}{\alpha}
\newcommand{\be}{\beta}
\newcommand{\ga}{\gamma}
\newcommand{\la}{\lambda}
\newcommand{\de}{\delta}
\newcommand{\eps}{\varepsilon}
\newcommand{\bx}{\bar x}
\newcommand{\by}{\bar y}

\newcommand{\iv}{^{-1} }

\newcommand {\R} {\mathbb R}
\newcommand {\N} {\mathbb N}

\newcommand {\B} {\mathbb B}
\newcommand {\Sp} {\mathbb S}
\newcommand {\gph} {{\rm gph}\,}%Graph
\newcommand {\dom} {{\rm dom}\,}

\newcommand {\cl} {{\rm cl}\,}
\newcommand {\co} {{\rm co}\,}

\newcommand{\TO}[1]{\stackrel{#1}{\to}}
\newcommand{\toto}{\rightrightarrows}% geht nur mit amssymb.sty
\def\es{\varnothing}

\def\LHS{left-hand side}

\def\SVM{set-valued mapping}

\def\Fr{Fr\'echet}
\newcommand{\norm}[1]{\left\Vert#1\right\Vert}

\newcommand{\magenta}[1]{\textcolor{magenta}{#1}}
\newcommand{\red}[1]{\textcolor{red}{#1}}

\newcommand{\ang}[1]{\left\langle #1 \right\rangle}
\newcommand{\qdtx}[1]{\quad\mbox{#1}\quad}
\newcommand{\AND}{\quad\mbox{and}\quad}
\newcounter{mycount}
\def\cnta{\setcounter{mycount}{\value{enumi}}}
\def\cntb{\setcounter{enumi}{\value{mycount}}}

\newcommand{\AK}[1]{\todo[inline]{AK {#1}}}
\newcommand{\rg}{{\rm rg}[F](\bx,\by)}
\newcommand{\srg}{{\rm srg}[F](\bx,\by)}
\newcommand{\lip}{{\rm lip}[F](\bx,\by)}
\newcommand{\clm}{{\rm clm}[F](\bx,\by)}

\newcommand\Rad[2]{{\rm rad[{#1}]}_{{\rm #2}}F(\bx,\by)}
\renewcommand{\rg}{{\rm{rg}}\,F(\bx,\by)}
\renewcommand{\srg}{{\rm{srg}}\,F(\bx,\by)}
\renewcommand{\lip}{{\rm{lip}}\,}
\renewcommand{\clm}{{\rm{clm}}\,}
\newcommand{\gF}{\tilde\gph F}
\newcommand{\ovclm}{fclm\,}
\newcommand{\HG}[1]{\todo[inline,color=green!40]{HG {#1}}}
\newcommand{\ssstar}{semismooth* }
\newcommand{\extclm}{firmly calm}
\newcommand{\email}[1]{{\tt #1}}

\newtheorem{theorem}{Theorem}[section]
\newtheorem{proposition}[theorem]{Proposition}
\newtheorem{remark}[theorem]{Remark}
\newtheorem{lemma}[theorem]{Lemma}
\newtheorem{corollary}[theorem]{Corollary}
\newtheorem{definition}[theorem]{Definition}
\newtheorem{example}[theorem]{Example}

\begin{document}
\title{
%Radii of Subregularity in Infinite Dimensions
Radius Theorems for Subregularity in Infinite Dimensions
\if{
\AK{25/04/22.
Shall we keep it plural?}
\thanks{
The second author benefited
from the support of the Australian Research Council, project DP160100854, the European Union's Horizon 2020 research and innovation
programme under the Marie Sk{\l }odowska--Curie Grant Agreement No. 823731
CONMECH.
%\HG{10/02/22. Is this ok now?}
}
}\fi
}
\author{
Helmut Gfrerer\thanks{
Institute of Computational Mathematics, Johannes Kepler University Linz, A-4040 Linz, Austria,
\email{helmut.gfrerer@jku.at},
ORCID 0000-0001-6860-102X
}
\and
Alexander Y. Kruger\thanks{Corresponding author. Centre for Informatics and Applied Optimization, Federation University, Ballarat, Australia;
%\email{a.kruger@federation.edu.au}\\
RMIT University, Melbourne, Australia,
%\email{alexander.kruger@rmit.edu.au}
\email{akrugeremail@gmail.com},
ORCID 0000-0002-7861-7380
}}

\maketitle

\vspace{-1cm}
\begin{center}\emph{Dedicated to the memory of our friend and colleague Professor Asen Dontchev}
\end{center}

\begin{abstract}
The paper continues our previous work \cite{DonGfrKruOut20} on the radius of subregularity
%in finite dimensions
that was initiated by Asen Dontchev.
We extend the results of \cite{DonGfrKruOut20} to general Banach/Asplund spaces and to other classes of perturbations, and sharpen the coderivative tools used in the analysis of the robustness of \emph{well-posedness} of mathematical problems and related \emph{regularity} properties of mappings involved in the statements.
We also expand the selection of classes of perturbations, for which the formula for the radius of strong subregularity is valid.
\end{abstract}



%\setcounter{tocdepth}{5}
%\tableofcontents

%\HG{2022/01/26. I must use the ''todonotes'' package instead of ''changes''}

\section{Introduction}
\if{
The paper continues our previous work \cite{DonGfrKruOut20} on the radius of subregularity in finite dimensions that was initiated by Asen Dontchev.
We extend the results of \cite{DonGfrKruOut20} to general Banach/Asplund spaces and to other classes of perturbations, and sharpen the coderivative tools used in the analysis of the robustness of \emph{well-posedness} of mathematical problems and related \emph{regularity} properties of mappings involved in the statements of the problems.
}\fi

The paper continues the line of ``radius of good behaviour'' theorems initiated by Dontchev, Lewis \& Rockafellar \cite{DonLewRoc03} in 2003, and aiming at quantifying the ``distance" from a given \emph{well-posed} problem to the set of ill-posed problems of the same kind.
Radius theorems go further than just establishing stability of a problem: they provide quantitative estimates of how far the problem can be perturbed before well-posedness is lost.
This is of significant importance, e.g., for computational methods.

Precursors of radius theorems can be traced back to the \emph{Eckart--Young theorem} \cite{EckYou36} dated 1936, which says that,
for any nonsingular $n \times n$ matrix $A$, it holds
\begin{gather*}
%\label{ek}
\inf_{B \in L(\R^n,\R^n)}\{\|B\| \mid A+B \mbox{ singular}\}
= \frac{1}{\|A^{-1}\|},
\end{gather*}
where $L(\R^n,\R^n)$ denotes the set of all $n\times n$ matrices, and $\|\cdot\|$ is the usual operator norm.
It gives the exact value for the distance from a given nonsingular square matrix to the set of all singular square matrices over the class of perturbations by arbitrary square matrices, with the distance being measured by the norm of a perturbation matrix.
This theorem  is connected with the  \emph{conditioning} of a matrix.
We refer the readers to the monograph \cite{BurCuc13} for a  broad coverage of the mathematics around \emph{condition numbers} and conditioning, and their role in numerical algorithms.

Far reaching generalizations of the Eckart--Young theorem were established in Dont\-chev, Lewis \& Rockafellar \cite{DonLewRoc03} for the fundamental property of \emph{metric regularity}, and then in Dontchev \& Rockafellar \cite{DonRoc04} for the important properties of \emph{strong metric regularity} and
\emph{strong metric subregularity} of set-valued mappings; see also \cite[Section~6A]{DonRoc14} and \cite[Section~5.4.1]{Iof17}.
The authors of \cite{DonLewRoc03,DonRoc04} considered perturbations over the classes of affine, Lipschitz continuous and calm single-valued functions, used Lipschitz and calmness moduli of the perturbation functions to measure the distance, and in finite dimensions established exact formulas for the radii in terms of the moduli of the respective regularity properties.
In general Banach spaces, lower bounds (providing sufficient conditions for the stability of the properties) were obtained
\cite{Iof01,NgaThe08} in terms of the respective regularity moduli as well as exact formulas in some particular cases \cite{DonLewRoc03,Mor04,CanDonLopPar05}.
The definitions of the mentioned regularity properties and the radius theorems from \cite{DonLewRoc03,DonRoc04} are collated in \cref{D1.1,T1.2}, respectively.

There have been several attempts to study stability of metric regularity with respect to set-valued perturbations under certain assumptions either of sum-stability or on the diameter of the image of perturbation mappings \cite{NgaTroThe14,AdlCibNga15,HeXu22}.
In line with the general pattern of radius theorems in \cite{DonLewRoc03,DonRoc04}, certain radius theorems for monotone mappings have been established recently in Dont\-chev, Eberhard \& Rockafellar~\cite{DonEbeRoc19}.

One thing strikes the eye of the reader of \cite{DonLewRoc03,DonRoc04}: the absence of any radius estimates for another fundamental regularity property of \SVM s, that of (not strong) \emph{metric subregularity} (see \cref{D1.1}(ii)).
On the contrary,
it was shown in \cite{DonRoc04} that the radius estimates similar to those in \cref{T1.2} do not hold for metric subregularity.
In particular, the class of affine functions is not appropriate for estimating the radius of subregularity, neither is the conventional subregularity modulus.
The main reason of the metric subregularity property being quite different and difficult to study lies in the well-known fact that, unlike its better studied siblings, it
%lacks robustness.
is not always stable.
%is often unstable.

The properties of metric subregularity and strong metric subregularity are key tools in analyzing convergence rates of numerical algorithms.
From the enormous number of papers on this subject we mention here two very recent ones \cite{LukThaTam18, YeYuaZenZha21}, where it is shown that metric subregularity at the solution yields linear convergence rates of several algorithms for solving generalized equations.
On the other hand, it is demonstrated in \cite{LukTebTha20} that metric subregularity is even necessary for the linear convergence of a certain class of algorithms.
Further, it is shown in \cite{CibDonKru18} that strong metric subregularity ensures superlinear convergence of the Josephy--Newton method \cite{Jos79} and some of its variants.

For a long time, studies of stability of metric subregularity have been limited to that of \emph{error bounds} of (mainly convex or almost convex) extended-real-valued functions.
(The latter property is equivalent to subregularity of the corresponding epigraphical multifunctions.)
Some
%limited results on stability issues of metric subregularity
sufficient and necessary conditions for the stability of error bounds
have been obtained
%for special types of mappings
in \cite{LuoTse94,ZheNg05,NgaKruThe10,ZheWei12, KruLopThe18,ZheNg19}.
We mention the \emph{radius of error bounds} formulas and estimates for several classes of perturbations in \cite
%[Theorems~2.2 and 3.2]
{KruLopThe18}.
Stability of metric subregularity for general \SVM s under smooth perturbations in finite dimensions has been studied in \cite{GfrOut16.2}.
In the very resent paper by Zheng \& Ng \cite{ZheNg21}, stability properties
%for special mappings and/or perturbations are considered
of metric subregularity is studied, mainly under the assumption that the perturbations are normally regular at the point under consideration.
In particular, the authors show that, for a
\SVM\ between Asplund spaces,
which is either metrically regular or strongly subregular at a reference point,
metric subregularity is stable with respect to small calm subsmooth perturbations.
In some cases, they provide radius-type estimates.
Some radius estimates for a special mapping defined by a system of linear inequalities have been established in \cite{CamCanLopPar}.

The fundamental results from \cite{DonLewRoc03,DonRoc04} motivated Dontchev, Gfrerer et al. in their recent paper \cite{DonGfrKruOut20}, restricted to finite dimensions, to pursue another approach and employ other tools when studying stability of metric subregularity.
Instead of the conventional subregularity modulus, several new ``primal-dual'' subregularity constants are used in \cite{DonGfrKruOut20} for estimating radii of subregularity.
Besides the standard class of Lipschitz continuous functions, semismooth and continuously differentiable perturbations are examined.
In the case of Lipschitz continuous perturbations, lower and upper bounds for the radius of subregularity are established, which differ by a factor of at most two.
The radii of subregularity over the classes of semismooth and continuously differentiable functions are shown to coincide, and the exact formula is obtained; see \cref{T1.3}.

In this paper, we extend the results of \cite{DonGfrKruOut20} to general Banach/Asplund spaces.
We consider the standard class of Lipschitz continuous perturbations as well as three new important for applications classes of functions: \emph{firmly calm} (see \cref{D1.2}), Lipschitz \emph{semismooth*} (see \cref{D1.3,D1.4}), and firmly calm semismooth*.
We also sharpen the primal-dual tools used in the analysis.

The motivation to study semismooth* perturbations stems from the successful application of the recently introduced semismooth* Newton method \cite{GfrOut21, GfrOut22} to generalized equations in finite dimensions; see \cite{GfrManOutVal, GfrOutVal}.
In a first attempt to generalize the semismooth* Newton method,  we carry over the notion of semismooth* sets and mappings to infinite dimensions and state some basic properties.

In our main \cref{T3.2}, we establish upper bounds for the radii of metric subregularity over all classes of perturbations, except Lipschitz semismooth*, in general Banach spaces, as well as lower bounds over all four classes in Asplund spaces.

In the case of Lipschitz continuous perturbations, the bounds differ by a factor of at most two.
As a byproduct, this gives a characterization of stability of subregularity under
small Lipschitz continuous perturbations.
In finite dimensions, the bounds reduce to those in \cite[Theorem~3.2]{DonGfrKruOut20}.
For firmly calm perturbations, in Asplund spaces we obtain
an exact formula for the radius, which however gives a positive radius only in exceptional cases.
This has motivated us to consider a more narrow class of firmly calm semismooth* perturbations, imposing additionally
a new {\em semismooth*} property.
For this class, we again have lower and upper bounds for the radius.

The proofs of the lower bounds for the radii of metric subregularity are based on the application of the conventional quantitative sufficient conditions for subregularity from  \cite{Gfr11,Kru15} (see \cref{T3.1}(ii)) coupled with a sum rule for set-valued mappings in \cref{T2.2}(ii).
The proofs of the upper bounds employ a rather non-conventional result in \cite[Theorem~3.2(2)]{Gfr11} (see \cref{T3.1}(i)), and are much more involved.

Using similar techniques, we briefly consider the property of strong metric subregularity, and establish in \cref{T3.3} an exact formula for the radius over the classes of calm, firmly calm and firmly calm semismooth* perturbations in general Banach spaces, thus, strengthening the correspondent assertion in \cite[Theorem~4.6]{DonRoc04}.

As mentioned above, we consider in this paper two new properties arising naturally in radius of subregularity considerations: \emph{firm calmness} and \emph{semismoothness*}.
The first property is defined for single-valued functions
%between normed spaces
and requires a function to be calm at a reference point and Lipschitz continuous around every other point nearby;
%except the reference point;
see \cref{D1.2}.
This property has the potential to be used outside the scope of the current topic to substitute the stronger local Lipschitz continuity property in some studies.
The semismoothness* property is defined for sets and \SVM s (see \cref{D1.3,D1.4}) as
an extension of the corresponding definitions in finite dimensions  introduced recently in \cite{GfrOut21}.
It is motivated by the formula for the radius of subregularity in \cite{DonGfrKruOut20}, where the conventional semismooth perturbations were considered.
The semismoothness* property is weaker than the
conventional semismoothness.
It plays an important role in \cite{GfrOut21} when constructing Newton-type methods for generalized equations.
Some characterizations of semismoothness* of sets and mappings as well as sufficient conditions ensuring the property are established.
In particular, it is shown that a positively homogenous function is semismooth*; see \cref{CorSS}.

The structure of the paper is as follows.
The next \cref{Pre} provides some preliminary material used throughout the paper.
This includes basic notation and general conventions, the definition of the new firm calmness property, definitions of the classes of perturbations typically used in stability analysis and corresponding radii, and certain primal-dual subregularity constants used in the radius estimates.
In \cref{S2}, we establish certain sum rules for mappings between normed (in most cases Asplund) spaces that are used in the sequel.
\cref{SS} is dedicated to new concepts of semismooth* sets and mappings, being infinite-dimensional extensions of the corresponding properties introduced recently in \cite{GfrOut21}.
In \cref{S3}, we formulate several new estimates and formulas for the radii of subregularity and strong subregularity, and introduce new primal-dual tools for quantitative characterization of subregularity of mappings and its stability, and potentially other related properties.
The proofs of the main results are in the separate \cref{S4}.

\section{Preliminaries}
\label{Pre}

\paragraph{Notation and basic conventions}

Our basic notation is standard; see \cite{RocWet98,Mor06.1,DonRoc14}.
Throughout the paper, if not explicitly stated otherwise, we assume that
$X$ and $Y$ are normed or, more specifically, Banach or Asplund
spaces.
Their topological duals are denoted by $X^*$ and $Y^*$, respectively, while $\langle\cdot,\cdot\rangle$ denotes the bilinear form defining the pairing between the spaces.
Recall that a Banach space is \emph{Asplund} if every continuous convex function on an open convex set is Fr\'echet differentiable on a dense subset, or equivalently, if the dual of each
%its
separable subspace is separable \cite{Phe93}.
%We refer the reader to \cite{Phe93,Mor06.1} for discussions about and characterizations of Asplund spaces.
All reflexive, particularly, all finite dimensional Banach spaces are Asplund.

We normally use the letters $x$ and $u$, often with subscripts, for elements of $X$, and the letters $y$ and $v$
for elements of $Y$.
Elements belonging to the corresponding dual spaces are marked with $*$ (i.e., $x^*$, $y^*$, etc.).
The open unit balls in a normed space and its dual are denoted by $\B$ and $\B^*$, respectively, while $\Sp$ and $\Sp^*$ stand for the unit spheres (possibly with a subscript denoting the space).
$B_\de(x)$ denotes the
%closed
{open}
ball with radius $\de>0$ and centre $x$.
Norms and distances in all spaces are denoted by the same symbols
$\|\cdot\|$ and $d(\cdot,\cdot)$, respectively.
$d(x,\Omega):=\inf_{\omega\in{\Omega}}\|x-\omega\|$ is the point-to-set distance from
$x$ to $\Omega$.

Symbols $\R$, $\R_+$ and $\N$ denote the sets of all real numbers, all nonnegative real numbers and all positive integers, respectively.
We use the following conventions:
$\inf\es_{\R}=+\infty$ and
$\sup\es_{\R_+}=0$, where $\es$ (possibly with a subscript) denotes the empty subset (of a given set).

If not specified otherwise, products of primal and dual normed spaces are assumed to be equipped with the sum and maximum norms, respectively:
\begin{gather}
\notag
\norm{(x,y)}=\norm{x}+\norm{y}, \quad (x,y)\in X\times Y,
\\
\label{dnorm}
\norm{(x^*,y^*)}=\max\{\norm{x^*},\norm{y^*}\}, \quad (x^*,y^*)\in X^*\times Y^*.
\end{gather}

We denote by $F:X\rightrightarrows Y$ a {\em
set-valued mapping} acting from $X$ to subsets of $Y$.
We write $f:X\to Y$ to denote a {\em single-valued
function}.
The graph and domain of
%a mapping
$F$ are defined as
$\gph F:= \{ (x,y)\in X\times Y\mid y \in F(x) \}$ and
$\dom F:= \{x \in X\mid F(x)\neq \es\}$, respectively.
The  inverse of $F$  is the mapping $y\mapsto  F^{-1}(y):=\{x\in X\mid y\in F(x)\}.$
$L(X,Y)$ denotes the space of all linear continuous maps $X\to Y$ equipped with the conventional operator norm.

%\paragraph{Lipschitz-like continuity properties}

\paragraph{Regularity, subregularity, Aubin property, and calmness}

The regularity and Lipschitz-like continuity properties of set-valued mappings in the definition below
are defined in metric terms.
They have been widely used in variational analysis and optimization, and well studied; cf. \cite{RocWet98,KlaKum02,DonLewRoc03,DonRoc04,Mor06.1, Gfr11,Gfr13, DonRoc14, Kru15,NgaPha15,DurStr16,Ude16,Zhe16,ZheZhu16,Iof17, CibDonKru18,Mar18, DonGfrKruOut20}.
For brevity, we drop the word ``metric'' from the names of the regularity properties.
For the definition of a single-valued localization of a \SVM\ we refer the readers to \cite{DonRoc14}.

\begin{definition}
\label{D1.1}
Let $F:X\rightrightarrows Y$ be a mapping between metric spaces, and $(\bx,\by)\in\gph F$.
\begin{enumerate}
\item
The mapping $F$ is
%metrically
regular at $(\bx,\by)$ if there exists an $\al>0$ such that
\begin{equation}
\label{D1.5-1}
\al d(x, F^{-1}(y)) \leq d(y, F(x))
      \quad \mbox{for all} \quad (x,y)\;\;\mbox{near}\;\;(\bx,\by).
\end{equation}
If, additionally, $F\iv$ has a single-valued localization around $(\by,\bx)$,
%that is Lipschitz continuous,
then
$F$ is
%called
strongly regular at $(\bx,\by)$.
\if{
\HG{09/05/22. I think it is enough to require: If, additionally, $F\iv$ has a single-valued localization at $(\by,\bx)$ then $F$ is ...}
\AK{18/5/22.
The existence of a (not necessarily continuous) single-valued localization automatically follows from \cref{D1.5-1}.
Am I missing something?}
\HG{25/5/22.
Consider the mapping $F:\R\toto\R$, $F(x)=x-\R_-$ at $(\bx,\by)=(0,0)$ which is obviously regular, but $F\iv$ does not have a single-valued localization. The Lipschitz continuity of a single-valued localization follows automatically from \cref{D1.5-1}, see also \cite[Proposition 3G.1]{DonRoc14}.}
}\fi
\item
The mapping $F$ is
%metrically
subregular at $(\bx,\by)$ if there exists an $\al>0$ such that
\begin{equation}
\label{D1.5-2}
\al d(x, F^{-1}(\by)) \leq d(\by, F(x))
      \quad \mbox{for all} \quad x\;\;\mbox{near}\;\;\bx.
\end{equation}
If, additionally,
$\bx$ is an isolated point of $F\iv(\by)$, then
$F$ is
%called
strongly subregular at $(\bx,\by)$.
\item
The mapping $F$ has the {Aubin property} at $(\bx,\by)$
if there exists an $\al>0$ such that
\begin{align}
\label{D1.1-1}
d(y,F(x))\le\al\;d(x,x')
\quad \mbox{for all} \quad x,x'\;\;\mbox{near}\;\;\bx\;\;\mbox{and}\;\; y\in F(x')\;\;\mbox{near}\;\;\by.
\end{align}
\item
The mapping $F$ is calm at $(\bx,\by)$
if there exists an $\al>0$ such that
\begin{align}
\label{D1.1-2}
d(y,F(\bx))\le\al\;d(x,\bx)
\quad \mbox{for all} \quad x\;\;\mbox{near}\;\;\bx\;\;\mbox{and}\;\; y\in F(x)\;\;\mbox{near}\;\;\by.
\end{align}
\end{enumerate}
\end{definition}

%The number $\al$ in \cref{D1.5-1,D1.5-2}
%provides a quantitative estimate of the respective property.
We denote the (possibly infinite) supremum of all $\al$ satisfying \cref{D1.5-1} (resp., \cref{D1.5-2}) by $\rg$ (resp., $\srg$), and call it the \emph{regularity}
(resp., \emph{subregularity}) \emph{modulus} of $F$ at $(\bx,\by)$.
Thus,
\begin{gather}
\label{rg}
\rg=\liminf_{(x,y)\to(\bx,\by),\;x\notin F\iv(y)}
\frac{d(y,F(x))}{d(x,F^{-1}(y))},
\\
\label{srg}
\srg=\liminf_{F^{-1}(\by)\not\ni x\to\bx}
\frac{d(\by,F(x))}{d(x,F^{-1}(\by))}.
\end{gather}
The case $\rg=0$ ($\srg=0$) indicates the absence of
regularity
(subregularity).

We denote the (possibly infinite) infimum of all $\al$ satisfying \cref{D1.1-1} (resp., \cref{D1.1-2}) by $\lip F(\bx,\by)$ (resp., $\clm F(\bx,\by)$), and call it the \emph{Lipschitz} (resp., \emph{calmness}) \emph{modulus} of $F$ at $(\bx,\by)$.
Thus,
\begin{gather*}
\lip F(\bx,\by)=\limsup_{\substack{x,x'\to\bx,\;x\ne x'\\ F(x')\ni y\to\by}}
\frac{d(y,F(x))}{d(x,x')},\quad
\clm F(\bx,\by)=\limsup_{\substack{\bx\ne x\to\bx\\ F(x)\ni y\to\by}}
\frac{d(y,F(\bx))}{d(x,\bx)}.
\end{gather*}

It is easy to see that
\begin{gather*}
(\rg)\iv=\lip F\iv(\by,\bx),\quad
(\srg)\iv=\clm F\iv(\by,\bx),
\end{gather*}
and $F$ is
regular (resp., subregular) at $(\bx,\by)$ if and only if $F\iv$ has the {Aubin property} (resp., is calm) at $(\by,\bx)$.

If $\bx$ is an isolated point of $F\iv(\by)$, formula \cref{srg} admits a simplification:
\begin{align}
\label{ssrg}
\srg=\liminf_{\bx\ne x\to\bx}
\frac{d(\by,F(x))}{d(x,\bx)}=
\liminf_{\bx\ne x\to\bx,\,y\in F(x)}
\frac{d(y,\by)}{d(x,\bx)}.
\end{align}
If the lower limit in \cref{ssrg} is zero, this indicates the absence of \emph{strong} subregularity.
Note that in this case the lower limit in \cref{srg} can still be nonzero (when $\bx$ is not an isolated point of $F\iv(\by)$), i.e. $F$ can be (not strongly) subregular at $(\bx,\by)$.
This is the only case when $\srg$ defined by \cref{srg} is not applicable for characterizing strong subregularity.

In the case of a single-valued function $f:X\to Y$, we write simply $\lip f(\bx)$ and $\clm f(\bx)$ to denote its Lipschitz and calmness moduli, i.e.,
$$
\lip f(\bx)= \limsup_{\substack{x,x'\to\bx,\,x\neq x'}}     \frac{d(f(x),f(x'))}{d(x,x')},
\quad
\clm f(\bx)= \limsup_{\substack{x\to\bx,\,x\neq\bx}}     \frac{d(f(x),f(\bx))}{d(x,\bx)}.
$$
%The {Aubin property} at a point obviously reduces in this case to the conventional Lipschitz continuity around this point.
If $\lip f(\bx)<+\infty$ (resp., $\clm f(\bx)<+\infty$), this indicates that $f$ is \emph{Lipschitz continuous} near $\bx$ (resp., \emph{calm} at $\bx$).
It obviously holds $0\le\clm f(\bx)\le\lip f(\bx)$.
Hence, if $f$ is \emph{Lipschitz continuous} near $\bx$, it is automatically \emph{calm} at $\bx$.
Simple examples show that
the converse implication does not hold in general.
If $f$ is affine, then obviously
$\clm f(\bx)=\lip f(\bx)$.

We are going to use a new continuity property lying strictly between the calmness and Lipschitz continuity.
It arises naturally when dealing with radius of subregularity estimates, and is likely to be of importance in other areas of analysis.
%\HG{09/05/22. I think, firm calmness is a continuity property and not a stability property.}

\begin{definition}
\label{D1.2}
A function $f:X\to Y$ is
\extclm\
at $\bx\in X$ if it is calm at $\bx$ and Lipschitz continuous around every $x\ne\bx$ near $\bx$.
\end{definition}

Note that a function which is \extclm\
at a point is not necessarily Lipschitz continuous near this point.
\if{
\AK{14/04/22.
The term \extclm\ has been suggested by Jiri.}
\AK{15/02/22.
I think this is an important property and is likely to be used in subsequent publications.
That is why I have decided to give a formal definition for possible future references.
However, it needs a better name.
First, I think grammatically it should be `extendedly calm', and this does not sound nice.
Second, an extension often means weakening of a property.
Possible options: `strongly calm', `supercalm', `weakly Lipschitz', `quasi-Lipschitz'.
None is perfect.
}
}\fi

\paragraph{Normal cones and coderivatives}

Dual estimates of the regularity radii require certain dual tools -- normal cones and coderivatives; cf. \cite{Kru81.1,Kru81.2,Kru03,Mor06.1}.

Given a subset $\Omega\subset X$, a point $\bx\in\Omega$, and a number $\eps\geq0$, the sets
\begin{gather}
\label{nc}
N_{\Omega,\eps}(\bx):=\left\{x^*\in X^*\,\Big |\,\limsup_{\Omega\ni x\to\bx,\; x\ne\bx}\frac{\ang{x^*,x-\bx}}{\norm{x-\bx}}\leq \eps\right\},
\\
\label{lnc}
\overline N_{\Omega}(\bx):=
\limsup_{\Omega\ni x\to\bx,\;\eps\downarrow0} N_{\Omega,\eps}(x)
\end{gather}
are called, respectively, the {\em set of \Fr\ $\eps$-normals} and the \emph{limiting normal cone} to $\Omega$ at $\bx$.
If $\eps=0$, the set \cref{nc} reduces to the \Fr\ normal cone $N_{\Omega}(\bx)$.
The $\limsup$ in \cref{lnc} is the sequential upper limit (in the sense of Painlev\'e--Kuratowski) with respect to the strong topology in $X$ and the weak* topology in $X^*$.
If $\Omega$ is convex, both $N_{\Omega}(\bx)$ and $\overline N_{\Omega}(\bx)$ coincide with the conventional normal cone in the sense of convex analysis.
If $X$ is Asplund, $\eps$ in \cref{lnc} can be dropped.
If $\overline N_\Omega(\bx)={N}_\Omega(\bx)$, then
$\Omega$ is said to be \emph{normally regular} at $\bx$.

Given a mapping $F:X\rightrightarrows Y$, a point $(\bx,\by)\in\gph F$, and a number $\eps\geq0$, the mappings $Y^*\toto X^*$ defined for all $y^*\in Y^*$ by
\begin{gather}
\label{cd}
D^*_\eps F(\bx,\by)(y^*):=\{x^*\in X^*\mid (x^*,-y^*)\in N_{\gph F,\eps}(\bx,\by)\},
\\
\label{lcd}
\overline D^*F(\bx,\by)(y^*):=\{x^*\in X^*\mid (x^*,-y^*)\in \overline N_{\gph F}(\bx,\by)\}
\end{gather}
are called, respectively, the {\em \Fr\ $\eps$-coderivative} and the \emph{limiting coderivative}
of $F$ at $(\bx,\by)$.
If $\eps=0$, the first one reduces to the {\em \Fr\ coderivative} ${D}^*F(\bx,\by)$.
In the case of a (single-valued) function $f:X\to Y$, we simply write $D^*_\eps f(\bx)(y^*)$ and $D^*f(\bx)(y^*)$ for all $y^*\in Y^*$.

Certain ``directional'' versions of \cref{lcd,lnc} have been introduced in \cite{Gfr11,Gfr13,DonGfrKruOut20}: for all $u\in X$ and $y^*\in Y^*$,
\begin{gather}
\notag
\overline N_{\Omega}(\bx;u):=\limsup_{\Omega\ni x\to\bx,\,t(x-\bx)\to u,\,t>0} N_{\Omega}(x),
\\
\label{dlcd}
\widehat{D}F(\bx,\by)(u,y^*):=\{(x^*,v)\in X^*\times Y\mid
(x^*,-y^*)\in\overline{N}_{\gph F} ((\bx,\by);(u,v))\}.
\end{gather}
They are
%(possibly empty) subsets of the corresponding sets in \cref{lcd,lnc}, and are
called, respectively, the {\em directional limiting normal cone}
to $\Omega$ at $\bar{x}$ in the direction $u\in X$,
and the {\em primal-dual derivative} of $F$ at $(\bx,\by)$.
%in the direction $(u,v)\in X\times Y$,
Observe that $\widehat{D}F(\bx,\by)$ acts from $X\times Y^*$ to $X^*\times Y$.
This explains the name.

It is well known \cite{Kru88,Mor06.1} that in Asplund spaces, when $\gph F$ is closed near $(\bx,\by)$, the regularity modulus \cref{rg} admits a dual representation:
\begin{align}
\label{rgd}
\rg=\liminf_{\substack{\gph F\ni(x,y)\to(\bx,\by)\\x^*\in{D}^*F(x,y)(S_{Y^*})}}\|x^*\|.
\end{align}
In finite dimensions, representation \cref{rgd} can be simplified:
\begin{align}
\label{rgd+}
\rg=\inf_{x^*\in\overline{D}^*F(\bx,\by)(S_{Y^*})}\|x^*\|.
\end{align}

Unlike the regularity modulus \cref{rg}, its subregularity counterpart \cref{srg} does not in general possess dual representations.
This is another reflection of the fact that the property lacks robustness.
Several dual (and primal-dual) subregularity constants have been used in \cite{Gfr11,Kru15,DonGfrKruOut20} to provide estimates of the subregularity modulus \cref{srg} as well as sufficient (and in some cases also necessary) conditions for subregularity.
None of them is in general equal to \cref{srg}.

\paragraph{Radius theorems}

%Following \cite{DonGfrKruOut20}, we aim at establishing estimates for the {\em radius of subregularity},  which quantifies the ``distance'' from a given subregular \SVM\ to the set of all mapping failing to have this property.

Below we formulate the key radius theorems from \cite{DonLewRoc03,DonRoc04,DonGfrKruOut20} which form the foundation for the results in this paper.
%
%To formulate the radius results from \cite{DonLewRoc03,DonRoc04}, we restrict ourselves to the case of a mapping $F:X\toto Y$ between normed spaces, and
We consider a mapping $F:X\toto Y$ between normed spaces and the following classes of perturbations of $F$ near a point $\bx\in X$:
\begin{align}
\label{Fclm}
\mathcal{F}_{clm}&:=\{f:X\to Y\mid f \mbox{ is calm at }\bx
%, \mbox{ and }f(\bx)=0
\},
\\
\label{Flip}
\mathcal{F}_{lip}&:=\{f:X\to Y\mid f \mbox{ is Lipschitz continuous around }\bx
%, \mbox{ and }f(\bx)=0
\},
\\
\label{FC1}
\mathcal{F}_{C^1}&:=\{f:X\to Y\mid f \mbox{ is }C^1\mbox{ around }\bx\},
\\
\label{Fss}
\mathcal{F}_{ss}&:=\{f:X\to Y\mid f \mbox{ is Lipschitz continuous around }\bx\mbox{ and semismooth at }\bx\},
\\
\label{Flin}
\mathcal{F}_{lin}&:=\{f:X\to Y\mid f \mbox{ is affine}\}.
\end{align}
We refer the readers to \cite{DonRoc14} for the definition of
\emph{semismooth} functions (in finite dimensions).
We obviously have
$\mathcal{F}_{lin}\subset\mathcal{F}_{C^1}\subset \mathcal{F}_{ss}\subset\mathcal{F}_{lip}\subset \mathcal{F}_{clm}$.
More classes are introduced in \cref{S3}.
Without loss of generality,
we will assume that perturbation functions $f$
in the above definitions
satisfy
$f(\bx)=0$.
(Thus, the functions in $\mathcal{F}_{lin}$ are actually linear.)
The
%corresponding
radii at a point $(\bx,\by)\in\gph F$ are defined as follows:
\begin{gather}
\label{rad}
\Rad{Property}{\mathcal{P}}:= \inf_{f\in\mathcal{F}_{\mathcal{P}}}\{\text{mod}\,f(\bx) \mid
F+f \mbox{
%does not obey
fails
the `Property' at } (\bx,\by)\}.
\end{gather}
Here, `Property' stands for `regularity', `subregularity', `strong regularity' or `strong subregularity'.
For brevity, we will write `R', `SR', `sR' and `sSR', respectively, in the notation of the radius.
$\mathcal{P}$ indicates the class of perturbations: \emph{clm, lip, $C^1$, ss} or \emph{lin}; more classes will be considered in \cref{S3}.
mod\,$f(\bx)$ identifies the modulus of $f$ at $\bx$ used in the computation of a particular radius: it can be either $\lip f(\bx)$ or $\clm f(\bx)$.
The first one is used when considering perturbations from $\mathcal{F}_{lip}$ or any its subclass, and the latter for $\mathcal{F}_{clm}$ and potentially other classes containing non-Lipschitz functions.
For instance, the definition of the radius of regularity over the class of Lipschitz continuous perturbations looks like this:
\begin{align*}
\Rad{R}{lip}:=& \inf_{f\in\mathcal{F}_{lip}}\{\lip f(\bx) \mid
F+f \mbox{ is not regular at } (\bx,\by)\}.
\end{align*}

The next theorem combines \cite[Theorem~1.5]{DonLewRoc03} and \cite[Theorems~4.6 and 5.12]{DonRoc04}.

\begin{theorem}
%[Radius theorem]
\label{T1.2}
Let $X$ and $Y$ be Banach spaces,
$F:X\rightrightarrows Y$ a mapping
{with closed graph},
%metrically subregular at
and $(\bx,\by)\in\gph F$.
Then
\begin{gather}
\label{T1.2-1}
\Rad{R}{lin}\geq
\Rad{R}{lip}\geq\rg,
\\
\label{T1.2-2}
\Rad{sSR}{lin}\geq
\Rad{sSR}{clm}\geq\srg.
\end{gather}
If $\dim X<\infty$ and $\dim Y<\infty$, then
\begin{gather}
\label{T1.2-3}
\Rad{R}{lin}=\Rad{R}{lip}=\rg,
\\
\label{T1.2-4}
\Rad{sSR}{lin}=\Rad{sSR}{clm}=\srg.
\end{gather}
Moreover, the equalities remain valid if $\mathcal{F}_{lin}$ is restricted to affine functions of rank $1$.

If $F$ is strongly regular at $(\bx,\by)$, then
%the above assertions
conditions \cref{T1.2-1,T1.2-3}
remain valid with $\Rad{sR}{lip}$ in place of $\Rad{R}{lip}$.
\sloppy
\end{theorem}

It is important to observe that the radii of regularity and strong regularity are considered in \cref{T1.2} with respect to Lipschitz continuous perturbations, and the regularity modulus is used in the estimates,
while in the case of strong subregularity, calm perturbations and the subregularity modulus are employed.
As observed in \cite[p.~2434]{ZheNg21}, it follows from \cite[Example~1E.5]{DonRoc14} that strong metric regularity is not stable with respect to small calm (even ``$0$-calm'') perturbations.
In fact, {it was shown in \cite[p.~2435]{ZheNg21} that} the perturbed mapping may fail to be even metrically regular.

In \cref{T3.3}, we show
%\todo{To be checked!}
that, when $F$ is strongly subregular at $(\bx,\by)$, the second equality in \cref{T1.2-4} holds in general Banach spaces.
At the same time, in infinite dimensions, the inequality $\Rad{R}{lin}\ge\rg$
in \cref{T1.2-1}
can be strict; cf. \cite[Theorem~5.61]{Iof17}.

%\paragraph{Radius of subregularity}

Note that \cref{T1.2} says nothing about the fundamental property of (not strong) subregularity, which turns out to be quite different.
%from the other three regularity properties from the point of view of its ``robustness''.
%This is not accidental: the latter property lacks robustness.
The next two examples show that it does not fit into the pattern of the conditions in \cref{T1.2}.
\if{
Similar to the definitions above, the radii of subregularity with respect to perturbations from the classes \cref{Fclm,Flip,Flin} can be defined as follows:
\begin{align}
\label{radSRclm}
\Rad{SR}{clm}:=& \inf_{f\in\mathcal{F}_{clm}}\{\clm f(\bx) \mid
F+f \mbox{ is not subregular at } (\bx,\by)\},
\\
\label{radSRlip}
\Rad{SR}{lip}:=& \inf_{f\in\mathcal{F}_{lip}}\{\lip f(\bx) \mid
F+f \mbox{ is not subregular at } (\bx,\by)\},
\\
\label{radSRlin}
\Rad{SR}{lin}:=& \inf_{f\in\mathcal{F}_{lin}}\{\lip f(\bx) \mid
F+f \mbox{ is not subregular at } (\bx,\by)\}.
\end{align}
Observe that $\lip f(\bx)$ in \cref{radSRlin} can be replaced with the calmness modulus $\clm f(\bx)$.
}\fi

\begin{example}
[Subregularity: perturbations from $\mathcal{F}_{lin}$]
\label{E1.1}
Let a function $F:\R\to\R$ be given by
$F(x)=x$ $(x\in X)$.
By \cref{D1.1}(i), $F$ is (strongly) regular, hence also (strongly) subregular everywhere, and
by \cref{rg,srg} (or \cref{ssrg}), rg\,$F(0,0)=\,$srg\,$F(0,0)=1$.
Any function $f\in\mathcal{F}_{lin}$ is of the form $f(x)=\la x$, where $\la\in\R$.
Hence, $(F+f)(x)=(1+\la)x$.
Thus, $F+f\in\mathcal{F}_{lin}$, and this function is (strongly) regular everywhere for any $\la\ne-1$, while with $\la=-1$, it is the zero function that is neither regular nor strongly subregular.
It follows that {\rm rad[R]}$_{{\rm lin}}F(0,0)=\,${\rm rad[sR]}$_{{\rm lin}}F(0,0)=\,${\rm rad[sSR]}$_{{\rm lin}}F(0,0)=1$, which agrees with \cref{T1.2}.
At the same time, the zero function is trivially subregular (with subregularity modulus equal $+\infty$);
%and, by \cref{radSRlin},
hence,
{\rm rad[SR]}$_{{\rm lin}}F(0,0)=+\infty$.
Thus, for subregularity an analogue of \cref{T1.2-3,T1.2-4} fails.
\end{example}

\begin{example}
[Subregularity: perturbations from $\mathcal{F}_{lip}$ and $\mathcal{F}_{clm}$]
\label{E1.2}
Let $F:\R\to\R$ be the zero function, i.e.,
$F(x)=0$ for all $x\in\R$.
As observed in \cref{E1.1}, $F$ is subregular everywhere, and {\rm srg}\,$F(0,0)=+\infty$.
Consider the function $f\in\mathcal{F}_{lip}$ given by $f(x)=x^2$ $(x\in X)$, and observe that it is not subregular at $(0,0)$, while $\lip f(0)=\clm f(0)=0$.
Hence,
%by \cref{radSRlip,radSRlin},
{\rm rad[SR]}$_{{\rm lip}}F(0,0)$=\,{\rm rad[SR]}$_{{\rm clm}}F(0,0)=0$.
Thus, for subregularity an analogue of \cref{T1.2-1,T1.2-2} fails.
\end{example}

In view of \cref{E1.1},
the class of perturbations $\mathcal{F}_{lin}$
%and definition \cref{radSRlin} do
does not seem appropriate for estimating the radius of subregularity.
\cref{E1.2} shows that,
if we want to use the natural classes of perturbations $\mathcal{F}_{lip}$ and $\mathcal{F}_{clm}$,
%and the corresponding definitions \cref{radSRlip,radSRclm}
it seems unlikely that the subregularity modulus $\srg$ can be used for estimating the radii of subregularity or, at least, the troublesome zero function (or, more generally, constant functions) should somehow be excluded.

It has been shown in our recent paper \cite{DonGfrKruOut20} that in finite dimensions certain ``primal-dual'' subregularity constants:
\begin{align}
\label{hatsrg}
\widehat{\rm srg}F(\bx,\by):= &\inf_{(x^*,v)\in\widehat{D}F(\bx,\by)(\Sp_{X\times Y^*})} \max\{\norm{v},\norm{x^*}\},
\\
\label{hatsrg+}
\widehat{\rm srg}^+F(\bx,\by):= &\inf_{(x^*,v)\in\widehat{D}F(\bx,\by)(\Sp_{X\times Y^*})} (\norm{v}+\norm{x^*}),
\\\notag
\widehat{\rm srg}^\circ F(\bx,\by):= &\inf_{\substack{(x^*,v)\in\widehat{D}F(\bx,\by)(u,y^*),\; \|u\|=\|y^*\|_*=1,
\\
B\in L(X,Y),\;B^*y^*=u^\ast,\; Bu=v}} \norm{B},
\end{align}
employing the primal-dual derivative $\widehat{D}F(\bx,\by)$ defined by \cref{dlcd},
can be used for estimating radii of subregularity in finite dimensions,
and one can consider additionally $C^1$ and {semismooth}
%(see, e.g., \cite{DonRoc14})
perturbations.
\if{
(with the standing assumption that $f(\bx)=0$), and the corresponding radii at $(\bx,\by)\in\gph F$:
\begin{align*}
%\label{radSRlip}
\Rad{SR}{ss}:=& \inf_{f\in\mathcal{F}_{ss}}\{\lip f(\bx) \mid
F+f \mbox{ is not subregular at } (\bx,\by)\},
\\
%\label{radSRlin}
\Rad{SR}{C^1}:=& \inf_{f\in\mathcal{F}_{C^1}}\{\lip f(\bx) \mid
F+f \mbox{ is not subregular at } (\bx,\by)\}.
\end{align*}
%For the original definition of semismoothness we refer to Mifflin  \cite{Mif77}.
}\fi
\if{
\AK{22/04/22.
Should we go into details about the differences between the the original definition of semismoothness in (Mifflin, 1977), the one in \cite{DonRoc14}, and the `generalization' in \cite{DonGfrKruOut20}?}

\HG{11/05/22. No, the definitions of semismoothness in \cite{DonRoc14} and our paper are equivalent}
}\fi

The next
%subregularity radius
statement is \cite[Theorem~3.2]{DonGfrKruOut20}.
It gives lower and upper estimates for the radius of subregularity with respect to Lipschitz continuous perturbations, and the exact formula for the radii with respect to semismooth and $C^1$ perturbations.

\begin{theorem}
%[Radius theorem]
\label{T1.3}
Let $\dim X<\infty$, $\dim Y<\infty$,
$F:X\rightrightarrows Y$
a mapping
{with closed graph},
%metrically subregular at
and $(\bx,\by)\in\gph F$.
Then
\begin{gather}
\label{T1.3-1}
\widehat{\rm srg}F(\bx,\by)\le\Rad{SR}{lip}\leq
\widehat{\rm srg}^+F(\bx,\by),
\\
\notag
\Rad{SR}{ss}=\Rad{SR}{C^1}=
\widehat{\rm srg}^\circ F(\bx,\by).
\end{gather}
\end{theorem}

Observe that in the case of a constant function
$X\ni x\mapsto F(x):=c\in Y$ we have
$\widehat{\rm srg}F(\bx,c)=
\widehat{\rm srg}^+ F(\bx,c)=
\widehat{\rm srg}^\circ F(\bx,c)=0$ for any $\bx\in X$.
Thus, each of the conditions $\widehat{\rm srg}F(\bx,\by)>0$ or $\widehat{\rm srg}^\circ F(\bx,\by)>0$ eliminates the troublesome constant functions, something that condition $\srg>0$ fails to do.
In fact, for the zero function in \cref{E1.2}, \cref{T1.3} gives
rad[SR]$_{{\rm lip}}F(0,0)$=\,rad[SR]$_{{\rm ss}}F(0,0)$=
rad[SR]$_{{\rm C^1}}F(0,0)=0$.
\sloppy

Further observe that $\widehat{\rm srg}F(\bx,\by)\leq \widehat{\rm srg}^+F(\bx,\by)\leq 2\widehat{\rm srg}F(\bx,\by)$. Thus, \cref{T1.3}
%provides us with
gives reasonably
tight upper and lower bounds for the radius of subregularity under Lipschitz continuous perturbations.
%We can also conclude that this
%The radius is positive
The property is stable
if and only if $\widehat{\rm srg}F(\bx,\by)$ (or $\widehat{\rm srg}^+F(\bx,\by)$) is positive.
\if{
In this paper, we consider the setting of general Banach or Asplund spaces, and the classes of calm \cref{Fclm} and Lipschitz continuous \cref{Flip} perturbations as well as several other classes.

The main radius estimates
%(some hold as equalities)
for subregularity and equalities for strong subregularity are collected in \cref{T6.1,T3.3}, respectively.
}\fi

\section{Sum rules for \SVM s}
\label{S2}

In this section, we establish certain sum rules for mappings between normed (in most cases Asplund) spaces.

The next statement is a \emph{fuzzy intersection rule} \cite[Lemma~3.1]{Mor06.1}.

\begin{lemma}
\label{L2.1}
Let $X$ be an Asplund space,
$\Omega_1,\Omega_2$ be closed subsets of $X$,
$\bx\in\Omega_1\cap\Omega_2$, and
$x^*\in N_{\Omega_1\cap\Omega_2}(\bx)$.
Then, for any $\eps>0$, there exist
$x_1\in\Omega_1\cap B_\eps(\bx)$,
$x_2\in\Omega_2\cap B_\eps(\bx)$,
$x_1^*\in N_{\Omega_1}(x_1)+\eps\B^*$,
$x_2^*\in N_{\Omega_2}(x_2)+\eps\B^*$, and
$\la\ge0$ such that
\begin{gather*}
\la x^*=x_1^*+x_2^*
\AND
\max\{\la,\|x_1^*\|\}=1.
\end{gather*}
\end{lemma}

The above lemma is instrumental in proving the intersection rule in the next theorem.
It is needed to prove the sum rule in \cref{T2.2}.

\begin{theorem}
\label{T2.1}
Let $X$ be an Asplund space,
$\Omega_1,\Omega_2$ be closed subsets of $X$,
$\bx\in\Omega_1\cap\Omega_2$, and
$x^*\in N_{\Omega_1\cap\Omega_2}(\bx)$.
The following assertions hold true.
\begin{enumerate}
\item
For any $\eps>0$, there exist
$x_1\in\Omega_1\cap B_\eps(\bx)$,
$x_2\in\Omega_2\cap B_\eps(\bx)$,
$x_1^*\in N_{\Omega_1}(x_1)$,
$x_2^*\in N_{\Omega_2}(x_2)$, and
$\la\ge0$ such that
\begin{gather}
\label{T2.1-1}
\|\la x^*-x_1^*-x_2^*\|<\eps
\AND
\max\{\la,\|x_1^*\|,\|x_2^*\|\}=1.
\end{gather}
\item
Suppose that
$\Omega_1,\Omega_2$ satisfy at $\bx$ the following \emph{normal qualification condition}:
there exist $\tau>0$ and $\de>0$ such that, for all
$x_1\in\Omega_1\cap B_\de(\bx)$,
$x_2\in\Omega_2\cap B_\de(\bx)$,
$x_1^*\in N_{\Omega_1}(x_1)$ and
$x_2^*\in N_{\Omega_2}(x_2)$,
it holds
\begin{gather}
\label{T2.1-2}
\|x_1^*+x_2^*\|\ge
\tau\max\{\|x_1^*\|,\|x_2^*\|\}.
\end{gather}
Then, for any $\eps>0$, there exist
$x_1\in\Omega_1\cap B_\eps(\bx)$ and
$x_2\in\Omega_2\cap B_\eps(\bx)$ such that
\begin{gather}
\label{T2.1-3}
x^*\in N_{\Omega_1}(x_1)+N_{\Omega_2}(x_2)+\eps\B^*.
\end{gather}
\end{enumerate}
\end{theorem}

\begin{proof}
\begin{enumerate}
\item
Let $\eps>0$.
Choose a positive number $\eps'<\min\{\eps/3,1\}$.
By \Cref{L2.1}, there exist
$x_1\in\Omega_1\cap B_\eps(\bx)$,
$x_2\in\Omega_2\cap B_\eps(\bx)$,
$u_1^*\in N_{\Omega_1}(x_1)+\eps'\B^*$,
$u_2^*\in N_{\Omega_2}(x_2)+\eps'\B^*$, and
$\la'\ge0$ such that
$\la'x^*=u_1^*+u_2^*$
and
$\max\{\la',\|u_1^*\|\}=1$.
Without loss of generality, we suppose that
$\max\{\la',\|u_1^*\|,\|u_2^*\|\}=1$.
Indeed, if $\|u_2^*\|>1$, we can replace $\la'$, $u_1^*$ and $u_2^*$ with $\la'/\|u_2^*\|$, $u_1^*/\|u_2^*\|$ and $u_2^*/\|u_2^*\|$, respectively.
There exist
$\hat u_1^*\in N_{\Omega_1}(x_1)$ and
$\hat u_2^*\in N_{\Omega_2}(x_2)$
such that
$\|\hat u_1^*-u_1^*\|<\eps'$ and
$\|\hat u_2^*-u_2^*\|<\eps'$.
Hence, $\|\la'x^*-\hat u_1^*-\hat u_2^*\|<2\eps'$.
Set $\al:=\max\{\la',\|\hat u_1^*\|,\|\hat u_2^*\|\}$.
If $\al\ge1$, we set $\la:=\la'/\al$, $x_1^*:=\hat u_1^*/\al$ and $x_2^*:=\hat u_2^*/\al$, and obtain
$\max\{\la,\|x_1^*\|,\|x_2^*\|\}=1$
and
$\|\la x^*-x_1^*-x_2^*\|<2\eps'<\eps$.
Let $\al<1$.
Then $\la'<1$ and
$\max\{\|u_1^*\|,\|u_2^*\|\}=1$.
Without loss of generality, $\|u_1^*\|\le\|u_2^*\|=1$.
Then $\|\hat u_2^*\|>1-\eps'>0$.
We set $\la:=\la'$, $x_1^*:=\hat u_1^*$ and $x_2^*:=\hat u_2^*/\|\hat u_2^*\|$.
Thus,
$\max\{\la,\|x_1^*\|,\|x_2^*\|\}=1$,
$\|x_2^*-\hat u_2^*\|=1-\|\hat u_2^*\|<\eps'$,
and consequently,
$\|\la x^*-x_1^*-x_2^*\|<3\eps'<\eps$.
\item
Let $\eps>0$.
If $x^*=0$, the conclusion holds true trivially.
Let $x^*\ne0$.
Set $\ga:=\max\{1,\|x^*\|/\tau\}$, and
choose a positive number $\eps'<\min\{\eps/(2\ga),\tau/2,\de\}$.
By (i), there exist
$x_1\in\Omega_1\cap B_{\eps'}(\bx)$,
$x_2\in\Omega_2\cap B_{\eps'}(\bx)$,
$x_1^*\in N_{\Omega_1}(x_1)$,
$x_2^*\in N_{\Omega_2}(x_2)$, and
$\la\ge0$ such that conditions \cref{T2.1-1} hold with $\eps'$ in place of $\eps$.
Then $\eps'<\eps/(2\ga)\le\eps/2<\eps$, and consequently,
$x_1\in B_{\eps}(\bx)$ and
$x_2\in B_{\eps}(\bx)$.
Moreover, using conditions \cref{T2.1-1,T2.1-2}, we obtain
\sloppy
\begin{align*}
\la\ga&=\max\{\la,\la\|x^*\|/\tau\}
\ge\max\{\la,(\la\|x^*\|+\eps')/\tau\}-\eps'/\tau
\\&
\ge\max\{\la,\|x_1^*+x_2^*\|/\tau\}-\frac12
\ge\max\{\la,\|x_1^*\|,\|x_2^*\|\}-\frac12=\frac12.
\end{align*}
Hence, $\|x^*-x_1^*/\la-x_2^*/\la\|<\eps'/\la\le2\ga\eps'<\eps$,
i.e. condition \cref{T2.1-3} is satisfied.
\end{enumerate}
\end{proof}

\begin{remark}
\begin{enumerate}
\item
It is easy to show that the conclusions of \cref{L2.1} and \cref{T2.1}(i) are equivalent.
\item
The normal qualification condition in \cref{T2.1}(ii) can be rewritten equivalently in the limiting form:
for any sequences
$\{x_{1k}\}\subset\Omega_1$ and
$\{x_{2k}\}\subset\Omega_2$
converging to $\bx$,
and
$\{x_{1k}^*\},\{x_{2k}^*\}\subset\B^*$ with
$x_{ik}^*\in N_{\Omega_i}(x_i)$ for $i=1,2$ and all $k\in\N$, it holds
\begin{gather*}
\|x_{1k}^*+x_{2k}^*\|\to0
\quad\Longrightarrow\quad
x_{1k}^*\to0\AND x_{2k}^*\to0.
\end{gather*}
In the case of closed sets in an Asplund space, it is equivalent to the \emph{limiting qualification condition} in \cite[Definition~3.2(ii)]{Mor06.1}.
\end{enumerate}
\end{remark}

The next theorem is a key tool for establishing radius of subregularity estimates.

\begin{theorem}
\label{T2.2}
Let $X$ and $Y$ be Asplund spaces,
$F_1,F_2:X\rightrightarrows Y$ be \SVM s with closed graphs,
$\bx\in X$,
$\by_1\in F_1(\bx)$,
$\by_2\in F_2(\bx)$,
$y^*\in Y^*$ and
$x^*\in D^*(F_1+F_2)(\bx,\by_1+\by_2)(y^*)$.
The following assertions hold true.
\begin{enumerate}
\item
For any $\eps>0$, there exist
$(x_1,y_1)\in\gph F_1\cap B_\eps(\bx,\by_1)$,
$(x_2,y_2)\in\gph F_2\cap B_\eps(\bx,\by_2)$,
$y_1^*,y_2^*\in Y^*$,
$x_1^*\in D^*F_1(x_1,y_1)(y_1^*)$,
$x_2^*\in D^*F_2(x_2,y_2)(y_2^*)$,
and
$\la\ge0$ such that
\begin{gather*}
%\label{T2.1-1}
\|(\la x^*-x_1^*-x_2^*,\la y^*-y_1^*,\la y^*-y_2^*)\|<\eps,
\\
\max\{\la,\|(x_1^*,y_1^*)\|,\|(x_2^*,y_2^*)\|\}=1.
\end{gather*}
\item
Suppose that $F_2$ satisfies the Aubin property near $(\bx,\by)$.
Then, for any $\eps>0$, there exist
$(x_1,y_1)\in\gph F_1\cap B_\eps(\bx,\by_1)$,
$(x_2,y_2)\in\gph F_2\cap B_\eps(\bx,\by_2)$, and
$y_1^*,y_2^*\in B_\eps(y^*)$
such that
\begin{gather}
\label{T2.2-3}
x^*\in D^*F_1(x_1,y_1)(y_1^*)+D^*F_2(x_2,y_2)(y_2^*)+\eps\B^*.
\end{gather}
\end{enumerate}
\end{theorem}

%\AK{9/11/21.
%Should part (i) be dropped?}

\begin{proof}
Observe that $X\times Y$ and $X\times Y\times Y$ are Asplund spaces, and define the closed sets
\begin{gather*}
\Omega_i:=\{(x,y_1,y_2)\in X\times Y\times Y\mid (x,y_i)\in\gph F_i\},\quad i=1,2.
\end{gather*}
Observe also that %$(\bx,\by_1,\by_2)\in\Omega_1\cap\Omega_2$
\begin{gather*}
\Omega_1\cap\Omega_2=\{(x,y_1,y_2)\in X\times Y\times Y\mid (x,y_1)\in\gph F_1,\; (x,y_2)\in\gph F_2\}
\end{gather*}
and, for all $(x,y_1,y_2)\in X\times Y\times Y$,
\begin{align*}
&N_{\Omega_1}(x,y_1,y_2)=\{(u^*,v^*,0)\in X^*\times Y^*\times Y^*\mid (u^*,v^*)\in N_{\gph F_1}(x,y_1)\},
\\
&N_{\Omega_2}(x,y_1,y_2)=\{(u^*,0,v^*)\in X^*\times Y^*\times Y^*\mid (u^*,v^*)\in N_{\gph F_2}(x,y_2)\},
\\
&N_{\gph(F_1+F_2)}(x,y_1+y_2)
\\
&\hspace{2cm}
\subset\{(u^*,v^*)\in X^*\times Y^*\mid (u^*,v^*,v^*)\in N_{\Omega_1\cap\Omega_2}(x,y_1,y_2)\}.
\end{align*}
Assertion (i) is now a direct consequence of \cref{T2.1}(i).

To prove assertion (ii), we first show that, thanks to the Aubin property of $F_2$, the sets
$\Omega_1,\Omega_2$ satisfy at $(\bx,\by,\by)$ the  normal qualification condition in \cref{T2.1}(ii).
Indeed, if $F_2$ satisfies the Aubin property, then, by \cref{D1.1}(iii) and the definition of the \Fr\ normal cone, there exist numbers $\al>0$ and $\de>0$ such that, for all $(x,y)\in\gph F_2\cap B_\de(\bx,\by)$ and all $(u^*,v^*)\in N_{\gph F_2}(x,y)$, it holds $\|u^*\|\le\al\|v^*\|$.
Hence, if
$(x_1,y_1,v_1)\in\Omega_1\cap B_\de(\bx,\by,\by)$,
$(x_2,v_2,y_2)\in\Omega_2\cap B_\de(\bx,\by,\by)$,
$(x_1^*,y_1^*,0)\in N_{\Omega_1}(x_1,y_1,v_1)$,
$(x_2^*,0,y_2^*)\in N_{\Omega_2}(x_2,v_2,y_2)$
with
$(x_1^*,y_1^*)\in N_{\gph F_1}(x_1,y_1)$ and
$(x_2^*,y_2^*)\in N_{\gph F_2}(x_2,y_2)$, then
\sloppy
\begin{align*}
{\max\{\|(x_1^*,y_1^*,0)\|,\|(x_2^*,0,y_2^*)\|\}}&=
\max\{\|x_1^*\|,\|x_2^*\|,\|y_1^*\|,\|y_2^*\|\}
\\
&\le
\max\{\|x_1^*+x_2^*\|+\|x_2^*\|,\|x_2^*\|,\|y_1^*\|,\|y_2^*\|\}
\\
&=
\max\{\|x_1^*+x_2^*\|+\|x_2^*\|,\|y_1^*\|,\|y_2^*\|\}
\\
&\le
\max\{\|x_1^*+x_2^*\|+\al\|y_2^*\|,\|y_1^*\|,\|y_2^*\|\}
\\
&\le
\max\{(\al+1)\max\{\|x_1^*+x_2^*\|,\|y_2^*\|\},\|y_1^*\|,\|y_2^*\|\}
\\
&\le
(\al+1)\max\{\|x_1^*+x_2^*\|,\|y_1^*\|,\|y_2^*\|\}
\\
&=
(\al+1)\|(x_1^*,y_1^*,0)+(x_2^*,0,y_2^*)\|,
\end{align*}
i.e. $\Omega_1,\Omega_2$ satisfy at $(\bx,\by,\by)$ the  normal qualification condition with $\tau:=(\al+1)\iv$.
By \cref{T2.1}(ii), there exist
$(x_1,y_1)\in\gph F_1\cap B_\eps(\bx,\by)$,
$(x_2,y_2)\in\gph F_2\cap B_\eps(\bx,\by)$,
$y_1^*,y_2^*\in Y^*$,
$x_1^*\in D^*F_1(x_1,y_1)(y_1^*)$, and
$x_2^*\in D^*F_2(x_2,y_2)(y_2^*)$ such that
\begin{gather*}
%\label{T2.1-1}
\|(x^*-x_1^*-x_2^*,y^*-y_1^*,y^*-y_2^*)\|<\eps,
\end{gather*}
i.e. $y_1^*,y_2^*\in B_\eps(y^*)$ and condition \cref{T2.2-3} is satisfied.
\end{proof}
\if{\AK{27/10/21.
Sum norm in primal product spaces and maximum norm in dual ones.}}\fi

The next statement complements \cref{T2.2}(ii).
It extends \cite[Theorem 1.62(i)]{Mor06.1} which addresses the case $\eps=0$.

\begin{theorem}
\label{T2.3}
Let
%$X$ and $Y$ be normed spaces,
$F:X\rightrightarrows Y$,
$f:X\to Y$,
$(\bx,\by)\in\gph F$, and $\eps\ge0$.
Suppose that $f$ is \Fr\ differentiable at $\bx$.
Set $\eps_1:=(\|\nabla f(\bx)\|+1)\iv\eps$ and $\eps_2:=(\|\nabla f(\bx)\|+1)\eps$.
Then, for all $y^*\in Y^*$, it holds
\begin{gather}
\label{T2.3-1}
D^*_{\eps_1} F(\bx,\by)(y^*)\subset
D^*_\eps(F+f)(\bx,\by+f(\bx))(y^*)-\nabla f(\bx)^*y^*\subset
D^*_{\eps_2} F(\bx,\by)(y^*).
\end{gather}
\end{theorem}

\begin{proof}
Let $y^*\in Y^*$ and $x^*\in D^*_{\eps_1}F(\bx,\by)(y^*)$.
By the definitions of the \Fr\ derivative, \Fr\ $\eps$-coderivative \cref{cd}, and the set of $\eps$-normals \cref{nc},
\begin{gather*}
\lim_{\substack{\bx\ne x\to\bx}}
\frac{f(x)-f(\bx)-\nabla f(\bx)(x-\bx)}{\norm{x-\bx}}=0,
\quad
\limsup_{\substack{\gph F\ni (x,y)\to(\bx,\by)\\ (x,y)\ne(\bx,\by)}}
\frac{\ang{x^*,x-\bx}-\ang{y^*,y-\by}}{\norm{(x,y)-(\bx,\by)}}
\leq \eps_1.
\end{gather*}
Hence,
\begin{align*}
\limsup_{\substack{\gph(F+f)\ni (x,y)\to(\bx,\by+f(\bx))\\ (x,y)\ne(\bx,\by+f(\bx))}}
&\frac{\ang{x^*+\nabla f(\bx)^*y^*,x-\bx}-\ang{y^*,y-\by-f(\bx)}} {\norm{(x,y)-(\bx,\by+f(\bx))}}
%\\=
%\limsup_{\substack{\gph F\ni (x,y)\to(\bx,\by)\\ (x,y)\ne(\bx,\by)}}
%\frac{\ang{x^*,x-\bx}-\ang{y^*,y-\by} -\ang{y^*,f(x)-f(\bx)-\nabla f(\bx)(x-\bx)}} {\norm{x-\bx}+\norm{(y-\by)+(f(x)-f(\bx))}}
\\=&
\limsup_{\substack{\gph F\ni (x,y)\to(\bx,\by)\\ (x,y)\ne(\bx,\by)}}
\frac{\ang{x^*,x-\bx}-\ang{y^*,y-\by}} {\norm{x-\bx}+\norm{y-\by+f(x)-f(\bx)}}
\\\le
(\|\nabla f(\bx)\|+1)
&\limsup_{\substack{\gph F\ni (x,y)\to(\bx,\by)\\ (x,y)\ne(\bx,\by)}}
\frac{\ang{x^*,x-\bx}-\ang{y^*,y-\by}}
{(\|\nabla f(\bx)\|+1)\norm{x-\bx}+\norm{y-\by+f(x)-f(\bx)}}
\\\le
(\|\nabla f(\bx)\|+1)
&\limsup_{\substack{\gph F\ni (x,y)\to(\bx,\by)\\ (x,y)\ne(\bx,\by)}}
\frac{\ang{x^*,x-\bx}-\ang{y^*,y-\by}}
{\norm{x-\bx}+\norm{y-\by}}\le(\|\nabla f(\bx)\|+1)\eps_1=\eps.
\end{align*}
This proves the first inclusion in \cref{T2.3-1}.
The second inclusion is a consequence of the first one applied with $F+f$, $-f$ and $\eps_2$ in place of $F$, $f$ and $\eps$, respectively.
\end{proof}

\begin{lemma}\label{LemEpsCoder}
Let
%$X$ and $Y$ be normed spaces,
$F:X\toto Y$, $f:X\to Y$, $(\bx,\by)\in\gph F$, and $\eps\ge0$.
Suppose that $f$ is calm at $\bx$ with $c:=\clm f(\bx)<1$.
Then, for all $y^*\in Y^*$ and $x^*\in{D^*_\eps}F(\bx,\by)(y^*)$, it holds $x^*\in D^*_\de(F+f)(\bx,\by+f(\bx))(y^*)$ with $\de:=(\eps+c\norm{y^*})/(1-c)$.
\if{
If $X$ and $Y$ are Asplund spaces, then for all $y^*\in Y^*$, $x^*\in D^*F(\bx,\by)(y^*)$, and $\eps>0$, there exist
$(\hat x,\hat y)\in\gph(F+f)\cap B_\eps(\bx,\by+f(\bx))$,
$\hat y^*\in Y^*$ and
$\hat x^*\in D^*F(\hat x,\hat y)(\hat y^*)$ such that
$\|(\hat x^*,\hat y^*)-(x^*,y^*)\|<\de+\eps$.
}\fi
\end{lemma}

\begin{proof}
Choose a number $c'\in (c,1)$.
For all $x$ near $\bx$ and all $y\in Y$, we have
\begin{align*}
\norm{(x-\bx,(y+f(x))-(\by+f(\bx)))}&=
\norm{x-\bx}+\norm{(y+f(x))-(\by+f(\bx))}
\\&\ge
\norm{x-\bx}+\norm{y-\by}-\norm{f(x)-f(\bx)}
\\&\ge
(1-c')\norm{x-\bx}+\norm{y-\by}
\\&\ge
(1-c')\norm{(x-\bx,y-\by)}.
\end{align*}
Hence,
\begin{align*}
\limsup_{\gph(F+f)\ni(x,y)\to(\bx,\by+f(\bx))} &\frac{\ang{x^*,x-\bx}-\ang{y^*,y-(\by+f(\bx))}}{\norm{(x-\bx,y-(\by+f(\bx)))}}
\\=
\limsup_{\gph F\ni(x,y)\to(\bx,\by)}
&\frac{\ang{x^*,x-\bx}-\ang{y^*,(y+f(x))-(\by+f(\bx))}} {\norm{(x-\bx,(y+f(x))-(\by+f(\bx)))}}
\\\le
\limsup_{\gph F\ni(x,y)\to(\bx,\by)}
&\frac{\ang{x^*,x-\bx}-\ang{y^*,y-\by}}{(1-c')\norm{(x-\bx,y-\by)}}
\\&+
\limsup_{(x,y)\to(\bx,\by)}
\frac{\norm{y^*}\norm{f(x)-f(\bx)}}{(1-c')\norm{x-\bx}}
\le
\frac{\eps+c'\norm{y^*}}{1-c'}.
\end{align*}
Taking infimum over all $c'\in (c,1)$, we arrive at the assertion.
\if{
the first assertion.
The second assertion is a consequence of the first one thanks to the standard representation of $\eps$-normals in Asplund spaces; cf. \cite[formula~(2.51)]{Mor06.1}.
}\fi
\end{proof}

\section{Semismooth* sets and mappings}
\label{SS}

\if{
\AK{18/05/22.
I think this deserves to be a separate section.
It is currently a little short.
Can anything be added?
}}\fi

We are going to employ an
infinite-dimensional extension of the \emph{semismoothness*} property introduced recently in \cite{GfrOut21}.
In finite dimensions, this property is weaker than the conventional semismoothness, and is important when constructing Newton-type methods for generalized equations.

\begin{definition}
\label{D1.3}
A set $\Omega\subset X$ is \ssstar at $\bx\in\Omega$ if,
for any $\eps>0$, there is a $\delta>0$ such that
\begin{equation}
\label{D1.3-1}
\vert \ang{x^*,x-\bx}\vert\leq \eps\norm{x-\bx}
\end{equation}
for all $x\in\Omega\cap B_\delta(\bx)$ and
$x^*\in N_{\Omega,\delta}(x)\cap \Sp_{X^*}$.
\end{definition}

\begin{theorem}
\label{T1.1}
Let $X$ be an Asplund space, and $\Omega\subset X$ be closed around $\bx\in\Omega$.
The set $\Omega$ is \ssstar at $\bx$ if and only if,
for any $\eps>0$, there is a $\delta>0$ such that condition \cref{D1.3-1} holds
for all $x\in\Omega\cap B_\delta(\bx)$ and
$x^*\in N_\Omega(x)\cap \Sp_{X^*}$.
\end{theorem}

\begin{proof}
Since for all $x\in\Omega$ and $\delta\geq0$ we have $N_\Omega(x)\subset N_{\Omega,\delta}(x)$,
%condition \cref{D6.1-1} implies condition \cref{T3.2-7}.
the ``only if'' part is true trivially.
We prove the ``if'' part by contradiction.
Suppose that $\Omega$ is not \ssstar at $\bx$, i.e. there are a number $\eps>0$ and sequences $x_k\in\Omega$ with $x_k\to\bx$, $\delta_k\downarrow0$, and $x_k^*\in N_{\Omega,\delta_k}(x_k)\cap \Sp_{X^*}$ such that
\[\vert\ang{x_k^*,x_k-\bx}\vert>\eps\norm{x_k-\bx}.\]
Then, for every $k\in\N$, there exist points $x_k'\in\Omega$ and $x_k'^*\in N_\Omega(x_k')$ satisfying $\|x_k'-x_k\|\leq \frac 1k\norm{x_k-\bx}$ and $\|x_k'^*-x_k^*\|\leq \delta_k+\frac 1k\norm{x_k-\bx}$, and $\|x_k'^*-x_k^*\|+\frac 1k \|x_k'^*\|<\eps$; cf. \cite[formula~(2.51)]{Mor06.1}.
Hence,
\begin{gather*}
\|x_k'-\bx\|\le\|x_k-\bx\|+\|x_k'-x_k\|\le \Big(1+\frac1k\Big)\|x_k-\bx\|,
\end{gather*}
and consequently,
\begin{align*}
\vert\ang{x_k'^*,x_k'-\bx}\vert&\geq \vert\ang{x_k^*,x_k-\bx}\vert- \|x_k'^*-x_k^*\| \norm{x_k-\bx}-\|x_k'^*\|\|x_k-x_k'\|
\\&>
\Big(\eps-\|x_k'^*-x_k^*\|-\frac 1k \|x_k'^*\|\Big)\norm{x_k-\bx}
\\&\ge
\Big(\eps-\|x_k'^*-x_k^*\|-\frac 1k \|x_k'^*\|\Big)\frac{k}{k+1}\|x_k'-\bx\|.
\end{align*}
Since $\|x_k^*\|=1$ and $\|x_k'^*-x_k^*\|\to0$ as $k\to\infty$, we can set $v_k^*:=x_k'^*/\|x_k'^*\|$ and conclude that
\[\left\vert\ang{v_k^*,x_k'-\bx}\right\vert >\frac\eps2\|x_k'-\bx\|\]
for all sufficiently large $k\in\N$.
%i.e. $\Omega$ is not \ssstar at $\bx$.
This completes the proof.
\end{proof}

\begin{remark}
\label{R1.1}
\begin{enumerate}
\item
Condition \cref{D1.3-1} in \cref{D1.3,T1.1} is obviously equivalent to the following one:
\begin{equation*}
%\label{D1.3-1}
\vert \ang{x^*,x-\bx}\vert\leq \eps\norm{x^*}\norm{x-\bx}.
\end{equation*}
Moreover, with the latter condition, the restriction $x^*\in\Sp_{X^*}$ in \cref{D1.3,T1.1} can be dropped.
 \item
\cref{D1.3} differs from the corresponding finite dimensional definition in \cite[Definition 3.1]{GfrOut21}.
Thanks to \cite[Proposition 3.2]{GfrOut21}, the two definitions are equivalent (in finite dimensions).
\end{enumerate}
\end{remark}

The class of \ssstar sets is rather broad.
E.g., it follows from \cite[Theorem 2]{Jou07} that in finite dimensions every closed subanalytic set is \ssstar at any of its points.

\begin{definition}
\label{D1.4}
A mapping $F:X\toto Y$ is \ssstar at $(\bx,\by)\in\gph F$ if $\gph F$ is \ssstar at $(\bx,\by)$.
\end{definition}

In view of \cref{R1.1}(i), we have the following explicit reformulation of the \ssstar property of a \SVM.

\begin{proposition}
\label{P1.1}
A mapping $F:X\toto Y$ is \ssstar at $(\bx,\by)\in\gph F$ if and only if,
for any $\eps>0$, there is a $\delta>0$ such that
\begin{align}
\label{P1.1-1}
\vert\ang{x^*,x-\bx}-\ang{y^*,y-\by}\vert\le\eps \norm{(x^*,y^*)}\norm{(x-\bx,y-\by)}
\end{align}
for all $(x,y)\in\gph F\cap B_\delta(\bx,\by)$, $y^*\in Y^*$, and $x^*\in D^*_\de F(x,y)(y^*)$.
\end{proposition}

\begin{proposition}\label{LemSS}
Let $f:X\to Y$
%be continuous,
and $\bx\in X$.
Suppose that the directional derivatives $f'(x;x-\bx)$ and $f'(x;\bx-x)$ exist
for all $x\ne\bx$ near $\bx$, and
\[ \lim_{\bx\ne x\to\bx}\frac {\max\{\|f(x)-f(\bx)-f'(x;x-\bx)\|,\|f(x)-f(\bx)+f'(x;\bx-x)\|\}} {\|(f(x)-f(\bx),x-\bx)\|}=0.\]
Then $f$ is \ssstar at $\bx$.
\if{
\HG{27/01/2017. If we do not require in Definition 2.2 that $\Omega$ is  closed then continuity of $f$ is not needed}
\AK{3/02/22.
Yes. And $U$ seems not needed either.}
}\fi
\end{proposition}

\begin{proof}
Let a number $\eps>0$ be fixed.
Choose a number $\delta>0$ such that
%$B_\delta(\bx)\subset U$,
$\delta(1+\eps/2)\leq\eps/2$, and
\begin{multline*}
\max\{\|f(x)-f(\bx)-f'(x;x-\bx)\|,\|(f(x)-f(\bx)+f'(x;\bx-x))\|\}
\\
<\frac\eps2\|(x-\bx,f(x)-f(\bx))\|
\qdtx{for all}x\in B_\delta(\bx)\setminus\{\bx\}.
\end{multline*}
Let $x\in B_\delta(\bx)\setminus\{\bx\}$ and $(x^*,y^*)\in N_{\gph f,\delta}(x,f(x))\cap\Sp_{(X\times Y)^*}$.
Then
\begin{gather*}
\frac{{\ang{x^*,x-\bx}+\ang{y^*,f'(x;x-\bx)}}} {\|(x-\bx,f'(x;x-\bx))\|}
=
\lim_{t\downarrow0} \frac{\ang{x^*,x+t(x-\bx)-x}+\ang{y^*,f(x+t(x-\bx))-f(x)}} {\|(x+t(x-\bx)-x,f(x+t(x-\bx))-f(x))\|}
\leq\delta.
\end{gather*}
Hence,
\begin{align*}
\ang{x^*,x-\bx}&+\ang{y^*,f(x)-f(\bx)}
\\&=
\ang{x^*,x-\bx}+\ang{y^*,f'(x;x-\bx)}+ \ang{y^*,f(x)-f(\bx)-f'(x;x-\bx)}
\\&\leq
\delta\|(x-\bx,f'(x;x-\bx))\|+\|f(x)-f(\bx)-f'(x;x-\bx)\|
\\&\leq
\delta\norm{(x-\bx,f(x)-f(\bx))}+ (\de+1)\|f(x)-f(\bx)-f'(x;x-\bx)\|
\\&\leq
(\delta+(\de+1)\eps/2)\norm{(x-\bx,f(x)-f(\bx))}
\leq \eps\norm{(x-\bx,f(x)-f(\bx))}.
\end{align*}
Similarly,
$\ang{x^*,\bx-x}+\ang{y^*,f(\bx)-f(x)}\le \eps\norm{(x-\bx,f(x)-f(\bx))}$,
and consequently,\\ $\vert\ang{x^*,x-\bx}+\ang{y^*,f(x)-f(\bx)}\vert\leq \eps\norm{(x-\bx,f(x)-f(\bx))}$.
\end{proof}

\begin{corollary}\label{CorSS}
If a function $f:X\to Y$ is positively homogenous at $\bx\in X$, i.e.,
\[f(\bx+\lambda(x-\bx))=f(\bx)+\lambda(f(x)-f(\bx))\qdtx{for all} x\in X,\;\lambda>0,\]
then it is \ssstar at $\bx$.
\end{corollary}

Let us now briefly compare the semismoothness* properties in \cref{D1.3,D1.4} with the corresponding extensions of smoothness and convexity due to Aussel, Daniilidis \& Thibault \cite{AusDanThi05}, and Zheng \& Ng \cite{ZheNg08.2,ZheNg21} known as \emph{subsmoothness} and \emph{$w$-subsmoothness}.
The definitions below employ the Clarke normal cones $N_{\Omega}^C(x)$.
We do not use this type of cones in the current paper and refer the readers to \cite{Cla83} for the respective definition and properties of such objects.

\begin{definition}
\label{D4.3}
Let $\Omega\subset X$ and $\bx\in\Omega$.
\begin{enumerate}
\item
$\Omega$ is
%called
subsmooth at $\bx$ if for any $\epsilon>0$ there is a $\delta>0$ such that
\begin{equation}
\label{EqSubSm}
\ang{x^*,u-x}\leq \epsilon\norm{u-x}
\end{equation}
for all $x,u\in\Omega\cap B_\delta(\bx)$ and
$x^*\in N_{\Omega}^C(x)\cap \B_{X^*}$.
\item
$\Omega$ is
%called
$w$-subsmooth at $\bx$ if for any $\epsilon>0$ there is a $\delta>0$ such that condition
\cref{EqSubSm} holds with
$x=\bx$
for all
$u\in\Omega\cap B_\delta(\bx)$ and
$x^*\in N_{\Omega}^C(\bx)\cap \B_{X^*}$.
\end{enumerate}
%Here, ${\overline N}_C(\Omega,x)$ denotes the Clarke normal cone to $\Omega$ at $\bx$.
\end{definition}

%Comparing the definition of semismoothness* with the one of w-subsmoothness, we observe that in \cref{EqSubSm} we fix $u=\bx$ for \ssstar sets and $x=\bx$ for w-subsmooth ones. Moreover, on the left hand side we take the absolute value $\vert \ang{x^*,\bx-x}\vert$ for defining the \ssstar property.
%Observe certain similarity between the definitions of semismoothness* in \cref{D1.3} and $w$-subsmoothness in \cref{D4.3}(ii):
%both correspond to fixing
%$\red{$u=\bx$}
%in \cref{EqSubSm}.
Observe the differences between the definitions  of semismoothness* in \cref{D1.3} and $w$-sub\-smoothness in \cref{D4.3}(ii): we fix $(u=)\bx$
%for \ssstar sets
in the first one,
and $x=\bx$
%for $w$-subsmooth ones in \red{Definition \ref{D4.3}(i)}.
in the latter.
%At the same time,
In addition,
condition \cref{D1.3-1} involves taking the absolute value in the \LHS.
Besides, in \cref{D4.3}(ii) (Clarke) normals are computed at $\bx$, while \cref{D1.3} requires computing (\Fr\ {$\de$-})normals at all
%nearby points.
$x$ near $\bx$.
\if{
\HG{13/06/2022.''both correspond to fixing $x=\bx$ in \cref{EqSubSm}.'' This is not true, because the inequality is not symmetric in $x$ and $u$ because of $x^*\in N_\Omega^C(x)$.}
\AK{14/06/22.
I think the wording is still OK.
To make the comparison more straightforward, I have swapped $x$ and $u$ in \cref{EqSubSm} (This is actually how they write it in \cite{ZheNg08.2}.) and in \cref{D4.3}(ii).}
\HG{14/06/22 I undid your changes in definition 4.3 (it now coincides with [42], where w-subsmoothness is defined). I read (4.3) as
\[\ang{x^*,u-x}\leq \epsilon\norm{u-x}\ \forall x,u\in\Omega\cap B_\delta(\bx)\ x^*\in N_{\Omega}^C(x)\cap \B_{X^*},\]
i.e., as the inequality together with the quantifiers.
Now fixing $x=\bx$ yields
\[\ang{x^*,u-\red{\bx}}\leq \epsilon\norm{u-\red{\bx}}\ \forall u\in\Omega\cap B_\delta(\bx)\ x^*\in N_{\Omega}^C(\red{\bx})\cap \B_{X^*}\]
(w-subsmooth sets) and fixing $u=\bx$ yields
\[\ang{x^*,\red{\bx}-x}\leq \epsilon\norm{\red{\bx}-x}\ \forall x\in\Omega\cap B_\delta(\bx)\ x^*\in N_{\Omega}^C(x)\cap \B_{X^*}\]
(semismooth* sets). In my opinion, from a geomtrical point of view, this is the essential difference. To avoid ambiguities I have now written \red{in Definition \ref{D4.3}(i)} instead of \red{in \eqref{EqSubSm}}.
 The difference that for semismooth* sets the Fr\'echet normal cone (in the Asplund setting) is used and for w-subsmooth sets the Clarke normal cone, is not so essential, because (4.1) remains valid for all $x^*\in \cl*\co \bar N_\Omega(x)$.
}}\fi

\begin{definition}
A mapping $F:X\toto Y$ is
%called
subsmooth ($w$-subsmooth) at $(\bx,\by)\in\gph F$ if
%its graph
$\gph F$ is subsmooth ($w$-subsmooth) at $(\bx,\by)$.
\end{definition}

Recall \cite{Mor06.1} that in a Banach space it holds ${N}^C_\Omega(x)\subset\cl^*\co \overline{N}_\Omega(x)$ for all $x\in\Omega$, and the inclusion holds as equality if the space is Asplund, and $\Omega$ is
locally closed around $x$. It follows immediately
%from the definition
that subsmoothness of a set $\Omega$ at $\bx\in\Omega$ implies that
\[\overline{N}_\Omega(x)\cap \B_{X^*}\subset {N}^C_\Omega(x)\cap \B_{X^*}\subset N_{\Omega,\epsilon}(x)\]
for any $\epsilon>0$ and all $x\in\Omega$ sufficiently close to $\bx$.
Furthermore, whenever a set $\Omega$ is subsmooth or  $w$-subsmooth at $\bx$, it is normally regular at $\bx$. On the other hand, a semismooth* set needs not to be normally regular at the point under consideration.
As an easy example consider the complementarity angle $\Omega=\{(x_1,x_2)\in\R^2_+\,\mid\, x_1x_2=0\}$.
It is obviously semismooth* at $(0,0)$ but not normally regular, hence, not ($w$-)subsmooth.

Since ($w$-)subsmoothness and the semismooth* properties for a \SVM\ are defined via the respective properties of its graph, similar considerations also apply to mappings.

\section{Radii of subregularity}
\label{S3}

In this section, $F:X\rightrightarrows Y$ is a mapping between normed spaces, and $(\bx,\by)\in\gph F$.

Several quantities employing ($\eps$-)coderivatives can be used for quantitative characterization of subregularity of mappings and its stability, and potentially other related properties.
Given an $\eps\ge0$, we are going to use the set
\begin{align}
\label{Del}
\Delta_\eps F:=&\{(x,y,x^*)\in X\times Y\times X^*\mid
{(x,y)\in\gph F},\;
x^*\in D^*_\eps F(x,y)(\Sp_{Y^*})\}
\end{align}
of the primal and dual (in $X^*$ corresponding to unit vectors in $Y^*$) components of the $\eps$-co\-derivative of $F$, as well as its projection on $X\times Y$:
\begin{align}
\label{Del0}
\Delta_\eps^{\circ}F:=&\{(x,y)\in\gph F\mid
D^*_\eps F(x,y)(\Sp_{Y^*})\ne\es\}.
\end{align}
When $\eps=0$, we will drop the subscript in \cref{Del,Del0}.
Given parameters $\eps>0$ and $\de\ge0$, we define local
analogues of \cref{Del,Del0}, involving the additional restriction \cref{P1.1-1} which arises when dealing with semismooth* perturbations; cf. \cref{P1.1}:
\begin{align}
\notag
\Upsilon_{\eps,\de}F(\bx,\by):=&\{(x,y,x^*)\mid
(x,y)\in\gph F,\;
x^*\in D^*_\de F(x,y)(y^*),\;y^*\in\Sp_{Y^*},
\\
\label{Ups}
&\quad
%\hspace{14mm}
\vert\ang{x^*,x-\bx}-\ang{y^*,y-\by}\vert\le\eps \norm{(x^*,y^*)}\norm{(x-\bx,y-\by)}\},
\\
\label{Ups0}
\Upsilon_{\eps,\de}^{\circ}F(\bx,\by):=&\{(x,y)\mid
(x,y,x^*)\in\Upsilon_{\eps,\de}F(\bx,\by)
\text{ for some } x^*\in X^*\}.
\end{align}
When $\de=0$, we will write $\Upsilon_{\eps}F(\bx,\by)$ and $\Upsilon_{\eps}^{\circ}F(\bx,\by)$.

Here is the list of regularity constants to be used in the sequel when estimating the radii of subregularity with respect to different classes of perturbations.
They employ notations \cref{Del,Ups,Del0,Ups0}.
\begin{align}
\label{srg1}
{\rm srg}_{1}F(\bx,\by):=&
\sup_{\eps>0}\;\inf_{\substack{(x,y,x^*)\in\Delta F,\; 0<\|x-\bx\|<\eps}}
\max\left\{\frac{\|y-\by\|}{\|x-\bx\|},\norm{x^*}\right\},
\\
\label{srg2}
{\rm srg}_{2}F(\bx,\by):=&
\sup_{\eps>0}\;\inf_{\substack{(x,y)\in\Delta^{\circ}F,\; 0<\|x-\bx\|<\eps}}
\frac{\|y-\by\|}{\|x-\bx\|},
\\
\label{srg3}
{\rm srg}_{3}F(\bx,\by):=&
\sup_{\eps>0}\;\inf_{\substack{(x,y,x^*)\in \Upsilon_{\eps}(\bx,\by),\; 0<\|x-\bx\|<\eps}}
\max\left\{\frac{\|y-\by\|}{\|x-\bx\|},\norm{x^*}\right\},
\\
\label{srg4}
{\rm srg}_{4}F(\bx,\by):=&
\sup_{\eps>0}\;\inf_{\substack{(x,y)\in \Upsilon_{\eps}^{\circ}(\bx,\by),\;0<\|x-\bx\|<\eps}}
\frac{\|y-\by\|}{\|x-\bx\|},
\\
\label{srg1+}
{\rm srg}_{1}^+F(\bx,\by):=&
\sup_{\eps>0}\;\inf_{\substack{(x,y,x^*)\in\Delta_\eps F,\; 0<\|x-\bx\|<\eps}}
\left(\frac{\|y-\by\|}{\|x-\bx\|}+\norm{x^*}\right),
\\
\label{srg2+}
{\rm srg}_{2}^+F(\bx,\by):=&
\sup_{\eps>0}\;\inf_{\substack{(x,y,x^*)\in\Delta_\de F,\; 0<\|x-\bx\|<\eps\\
0<\de<\eps,\;\de\|x^*\|<\eps}}
\frac{\|y-\by\|}{\|x-\bx\|},
\\
\label{srg4+}
{\rm srg}_{4}^+F(\bx,\by):=&
\sup_{\eps>0}\;\inf_{\substack{(x,y,x^*)\in \Upsilon_{\eps,\de}F(\bx,\by),\;0<\|x-\bx\|<\eps\\
0<\de<\eps,\; \de\|x^*\|<\eps}}
\frac{\|y-\by\|}{\|x-\bx\|}.
\end{align}

\begin{remark}
\label{R3.1}
\begin{enumerate}
\item
The $\sup\inf$ constructions in \cref{srg1,srg3,srg4,srg2,srg1+,srg4+,srg2+} can be rewritten as $\liminf$.
We have chosen this explicit form to avoid confusion with the common usage of $\liminf$ in set-valued analysis.
%\AK{23/02/22.
%$\liminf$ is not bad actually. :-)}

\item
Thanks to the term $\|y-\by\|/\|x-\bx\|$ present in all the definitions \cref{srg1,srg3,srg4,srg2,srg1+,srg4+,srg2+},
%$x\to\bx$ in all the limits can be replaced with $(x,y)\to(\bx,\by)$.
the inequality ${\|x-\bx\|<\eps}$ in these definitions can be replaced with the stronger one: ${\|(x,y)-(\bx,\by)\|<\eps}$.

\item
In view of the definition \cref{dnorm} of the dual norm on the product space, $\|(x^*,y^*)\|$ in \cref{Ups} equals $\max\{\|x^*\|,1\}$.

\item
In
%\cref{srg3,srg4+},
\cref{srg3},
the set $\Upsilon_{\eps}(\bx,\by)$ can be replaced with its slightly simplified version with the inequality in the definition \cref{Ups} substituted by the next one:
\begin{align*}
\ang{x^*,x-\bx}-\ang{y^*,y-\by}\vert\le\eps \norm{(x-\bx,y-\by)}.
\end{align*}
%\HG{14/02/22. I think this does not longer hold}
\end{enumerate}
\end{remark}
\if{
\AK{6/01/22.
Each constant characterizes something stronger than subregularity.
Could it make sense to check what exactly each constant actually characterizes?
Can they be meaningful by themselves apart from producing estimates for the radius?
}
\AK{10/01/22.
The system of notations srg$_1$, srg$_2$,\ldots\ does not look nice.
I hesitate to start writing srg$_{lip}$, srg$_{clm}$,\ldots\ because the above regularity constants may potentially be applicable to characterizing other properties, not just the radii.
The list of constants may also grow bigger.
}
}\fi

\begin{proposition}
\label{P3.1}
Let
%$X$ and $Y$ be
%Banach
%{normed}
%spaces,
$F:X\rightrightarrows Y$ and
$(\bx,\by)\in\gph F$.
\begin{enumerate}
\item
${\rm srg}_{2}F(\bx,\by)\le{\rm srg}_{1}F(\bx,\by)\le
{\rm srg}_{3}F(\bx,\by)$;
\item
${\rm srg}_{2}F(\bx,\by)\le
{\rm srg}_{4}F(\bx,\by)\le{\rm srg}_{3}F(\bx,\by)$;
\item
${\rm srg}_{2}^+F(\bx,\by)\le{\rm srg}_{4}^+F(\bx,\by)$.
\cnta
\end{enumerate}
Suppose that $X$ and $Y$ are Asplund spaces.
Then
\begin{align}
\label{Eq_srg1+Asp}{\rm srg}_{1}^+F(\bx,\by)=&
\sup_{\eps>0}\;\inf_{\substack{(x,y,x^*)\in\Delta F,\; 0<\|x-\bx\|<\eps}}
\left(\frac{\|y-\by\|}{\|x-\bx\|}+\norm{x^*}\right),
\\
\label{Eq_srg2+Asp}
{\rm srg}_{2}^+F(\bx,\by)=&
\sup_{\eps>0}\;\inf_{\substack{(x,y)\in\Delta^\circ F,\; 0<\|x-\bx\|<\eps}}
\frac{\|y-\by\|}{\|x-\bx\|},
\\
\label{Eq_srg4+Asp}
{\rm srg}_{4}^+F(\bx,\by)=&
\sup_{\eps>0}\;\inf_{\substack{(x,y)\in \Upsilon^\circ_{\eps}F(\bx,\by),\;0<\|x-\bx\|<\eps}}
\frac{\|y-\by\|}{\|x-\bx\|}.
\end{align}
Moreover,
\begin{enumerate}
\cntb
\item
${\rm srg}_{1}F(\bx,\by)\le{\rm srg}_{1}^+F(\bx,\by){\leq 2{\rm srg}_{1}F(\bx,\by)}$;
\item
${\rm srg}_{2}F(\bx,\by)=
{{\rm srg}_{2}^+F(\bx,\by)}$;
\item
${\rm srg}_{4}F(\bx,\by)\le{\rm srg}_{4}^+F(\bx,\by)$.
\cnta
\end{enumerate}
Suppose that $\dim X<+\infty$ and $\dim Y<+\infty$.
\begin{enumerate}
\cntb
\item
${\rm srg}_{1}F(\bx,\by)=\widehat{\rm srg}F(\bx,\by)$;
\item
${\rm srg}_{1}^+F(\bx,\by)=\widehat{\rm srg}^+F(\bx,\by)$,
\end{enumerate}
where $\widehat{\rm srg}F(\bx,\by)$ and $\widehat{\rm srg}^+F(\bx,\by)$ are defined by \cref{hatsrg,hatsrg+}, respectively.
\end{proposition}

\begin{proof}
Assertions (i)--(iii) follow immediately from the definitions of the respective constants.
To justify \eqref{Eq_srg1+Asp}, first note that the Fr\'echet $\eps$-coderivative is defined by Fr\'echet $\eps$-normals (see \cref{cd}).
Second, given a closed subset $\Omega$ of an Asplund space, and any $\bx\in\Omega$, $\epsilon>0$ and $x_\eps^*\in N_{\Omega,\eps}(\bx)$, there are $x\in\Omega\cap B_\eps(\bx)$ and $x^*\in N_\Omega(x)$ satisfying $\norm{x^*-x^*_\eps}<2\eps$; cf. \cite[formula (2.51)]{Mor06.1}.
Hence, $\Delta_\eps F$ in \cref{srg1+} can be replaced with $\Delta F$, yielding \eqref{Eq_srg1+Asp}.
The same arguments show that, in the Asplund space setting, one can set $\delta=0$ in \cref{srg2+,srg4+}, yielding \cref{Eq_srg2+Asp,Eq_srg4+Asp}.
Note that in the case of \cref{Eq_srg2+Asp},  conditions $(x,y,x^*)\in\Delta_0 F$ and $0\cdot\norm{x^*}<\eps$ are equivalent to $(x,y)\in \Delta^\circ F$.
Having established \cref{Eq_srg1+Asp}--\cref{Eq_srg4+Asp}, the estimates (iv)--(vi) easily follow.
Finally, to verify (vii) and (viii), consider  suitable sequences $(x_k,y_k)\to(\bx,\by)$ and $(x_k^*,y_k^*)$ with $x_k^*\in D^*F(x_k,y_k)(y_k^*)$, $y_k^*\in\Sp_{Y^*}$ and
\[{\rm srg}_{1}F(\bx,\by)=\lim_{k\to\infty}\max \left\{\frac{\norm{y_k-\by}}{\norm{x_k-\bx}}, \norm{x_k^*}\right\}\quad\left(\text{or}\quad{\rm srg}_{1}^+F(\bx,\by)=\lim_{k\to\infty}\frac{\norm{y_k-\by}}{\norm{x_k-\bx}}+\norm{x_k^*}\right).\]
Since both $X$ and $Y$ are finite dimensional, by passing to subsequences without relabelling, we can assume that
\begin{equation}\label{EqAuxProp5.2}
\left(\frac{x_k-\bx}{\norm{x_k-\bx}}, \frac{y_k-\bx}{\norm{x_k-\bx}}, x_k^*,y_k^*\right)\to (u,v,x^*,y^*).
\end{equation}
It follows that $(u,y^*)\in \Sp_{X\times Y^*}$ and $(x^*,v)\in\widehat D^*F(\bx,\by)(u,y^*)$, and consequently,
\[\widehat{\rm srg}F(\bx,\by)\leq\max\{\norm{u},\norm{x^*}\}={\rm srg}_{1}F(\bx,\by)\quad\left(\text{or}\quad\widehat{\rm srg}^+F(\bx,\by)\leq\norm{u}+\norm{x^*}={\rm srg}_{1}^+F(\bx,\by)\right).\]
The opposite inequalities are valid in arbitrary normed spaces and rely on the fact that, by the definition of $\widehat D F$, for any $(u,y^*)\in \Sp_{X\times Y^*}$ and $(x^*,v)\in\widehat D^*F(\bx,\by)(u,y^*)$, there are sequences $(x_k,y_k)\to (\bx,\by)$  and $(x_k^*,y_k^*)$ with $x_k^*\in D^*F(x_k,y_k)(y_k^*)$, $y_k^*\in\Sp_{Y^*}$ verifying \eqref{EqAuxProp5.2}.
%We omit the details.
%These arguments show that we can work with Fr\'echet coderivatives in Asplund spaces.
\end{proof}

\if{
\AK{25/04/22.
What about $\widehat{\rm srg}^\circ F(\bx,\by)$?}
\HG{11/05/22.
I think we cannot say anything about $\widehat{\rm srg}^\circ F(\bx,\by)$ because its construction is tailored to the finite-dimensional case.}
}\fi
\if{
\AK{23/02/22.
The equality in (v) needs to be proved.}
\AK{1/01/22.
$\|(x^*,y^*)\|$ in the definition of
${\rm srg}_{4}^+F(\bx,\by)$ (and in the other definitions?) can be dropped.
}
\AK{6/01/22.
$\|(x^*,y^*)\|$ removed from the definition of
${\rm srg}_{4}^+F(\bx,\by)$.
It can be removed from the definition of
${\rm srg}_{3}F(\bx,\by)$.
I am not sure about
${\rm srg}_{4}F(\bx,\by)$ unless condition $\|x^*\|<M$ is added.
}
\HG{28/01/2022. I reformulated the definitions of the constants: Before it had the form $\liminf_{x\to\bx} S(x)$, where $S(x)$ is a set of real numbers, and it is not clear for me what this means in this context.
\if{
I agree with you that  $\|(x^*,y^*)\|$ could be removed from the definitions of ${\rm srg}_{4}^+F(\bx,\by)$ and ${\rm srg}_{3}F(\bx,\by)$. However, I added it again to the definitions because we cannot remove it from the definition of ${\rm srg}_{4}F(\bx,\by)$ and in my opinion all the formulas should be built by using the same elements in order to increase readability. Another reason is that the definition of the semismoothness* property involves only normalized normals and this is reflected in the formula by $\|(x^*,y^*)\|$.\\
I am also aware of the fact that $\liminf_{(x,y)\TO{\gph F}(\bx,\by)}$ could be replaced by $\liminf_{x\to \bx}$ and $y\in F(x)$ due to the term $\frac{\|y-\by\|}{\|x-\bx\|}$. Again for better readability (at least for me), I used $\liminf_{(x,y)\TO{\gph F}(\bx,\by)}$.
}\fi}

\AK{5/02/22.
I've reformulated the definitions again making them more compact.
I've tried to address your concerns (as well as those of the potential readers) in the subsequent \cref{R3.1}.
\if{
The $\gF$ notation does not save much space.
Besides, it is likely to confuse the readers, unless they are reminded the definition many times.
}\fi
Is the new version acceptable?
}
}\fi
\if{
\AK{27/11/21.
Can ${\rm srg}_0F(\bx,\by)$ and ${\rm srg}_{3}F(\bx,\by)$ be compared?}
\HG{28/01/2022. I do not think so. For \ssstar mappings $F$ there should hold ${\rm srg}_{1}F(\bx,\by)={\rm srg}_{3}F(\bx,\by)$, but in the example we have ${\rm srg}_0F(\bx,\by)<{\rm srg}_{3}F(\bx,\by)$}
}\fi

The next theorem contains characterizations of subregularity from \cite{Gfr11}, which play a key role when establishing radii estimates.
Note that condition ${\rm srg}_{1}^+\,F(\bx,\by)=0$ in part~(i) does not in general imply the absence of subregularity; consider the zero function in \cref{E1.2}.

\begin{theorem}
\label{T3.1}
Let $X$ and $Y$ be Banach spaces,
$F:X\rightrightarrows Y$, and
$(\bx,\by)\in\gph F$.
\begin{enumerate}
\item
If ${\rm srg}_{1}^+\,F(\bx,\by)=0$, then there exists a $C^{1}$ function $f:X\to Y$ with $f(\bx)=0$ and
${\nabla f(\bx)=0}$ such that $F+f$ is not
%metrically
subregular at $(\bx,\bar{y})$.
\item
Suppose that $X$ and $Y$ are Asplund spaces, and $\gph F$ is closed.
If ${\rm srg}_{1}F(\bx,\by)>0$,
then $F$ is subregular at $(\bx,\by)$.
\end{enumerate}
\end{theorem}
\begin{proof}
Assertion~(i) is a consequence of \cite[Theorem~3.2(2)]{Gfr11}.
Assertion~(ii) follows from the Asplund space part of \cite[Theorem~3.2(1)]{Gfr11}
after observing that condition ${\rm srg}_{1}F(\bx,\by)>0$ is equivalent to $(0,0)\notin{\rm Cr}_0(\bx,\by)$, where
\[{\rm Cr}_0(\bx,\by):=\left\{(v, x^*)\in Y\times X^*\;\Big\vert\; \begin{matrix}\exists t_k\downarrow 0,\ (v_k,x_k^*)\to (v,x^*),\
 (u_k,y_k^*)\in \Sp_{X}\times\Sp_{Y^*}\\\mbox{ with }x_k^*\in DF^*(\bx+t_ku_k,\by+t_kv_k)(y_k^*)\end{matrix}\right\}\] is the {\em limit set critical for
%metric
subregularity} \cite{Gfr11}.
\end{proof}
\begin{remark}
%\label{R3.2}
Using \cite[Corollary~5.8]{Kru15} with condition (g), one can strengthen \cref{T3.1}(ii) and show that
${\rm srg}_{1}F(\bx,\by)\le\srg$.
%Note that in the Asplund space setting the condition ${\rm srg}_{1}F(\bx,\by)>0$ is not the weakest possible sufficient condition for subregularity. We refer to \cite{Kru15} for a lot of other sufficient conditions.
\sloppy
\end{remark}

%\section{Radius}
%\label{S5}

%In this section, we assume a \SVM\ $F:X\rightrightarrows Y$ between normed spaces and a point $(\bx,\by)\in\gph F$ to be fixed, and
In addition to $\mathcal{F}_{clm}$ and $\mathcal{F}_{lip}$ defined by \cref{Fclm,Flip}, respectively,
we are going to consider the following three classes of perturbations of $F$ near $\bx$:
\begin{align*}
\mathcal{F}_{\ovclm}&:=\left\{f:X\to Y\mid
f\mbox{ is \extclm\ at }\bx
\right\},
\\
\mathcal{F}_{lip+ss^*}&:=\{f:X\to Y\mid f \mbox{ is Lipschitz continuous around }\bx \mbox{ and semismooth* at }\bx\},
\\
\mathcal{F}_{\ovclm+ss^*}&:=\left\{f:X\to Y\mid
f\mbox{ is \extclm\ and \ssstar at }\bx
\right\}.
\end{align*}
\if{
\magenta{The elements of the class $\mathcal{F}_{\ovclm}$ are called {\em extended calm} functions. Note that a extended calm function needs not to be Lipschitzian around $\bx$, i.e., we can have $\limsup_{x\to\bx}\lip f(x)=\infty$.}
\HG{10/02/22. I think the notation $clm+lip$ is misleading because the functions need not to be Lipschitzian  near $\bx$. Maybe you have a better proposal than {\em ext.clm}: Something which indicates that we have just a little bit more than calmness.}
}\fi
\if{
\begin{remark}
The functions belonging to $\mathcal{F}_{\ovclm}$ are called {\em extended calm}.
Note that an extended calm function needs not to be Lipschitz continuous around $\bx$, i.e., we can have $\lip f(\bx)=+\infty$.
\end{remark}
\AK{12/02/22.
Has this terminology been used?
}
\HG{14/02/22. No. Please feel free to modify it. But I think that it is important to emphasize that we need a little bit more than calmness}
}\fi

The next
%inclusion is
{relations are} immediate from the definitions:
\begin{gather*}
\mathcal{F}_{lip+ss^*}=\mathcal{F}_{lip}\cap
\mathcal{F}_{\ovclm+ss^*},
\quad
{\mathcal{F}_{lip}\cup\mathcal{F}_{\ovclm+ss^*}\subset \mathcal{F}_{\ovclm}.}
\end{gather*}
We now formulate stability estimates for the property of
%metric
subregularity with respect to these
classes of perturbations.
Without loss of generality,
we will assume that perturbation functions $f$
in the above definitions
satisfy
$f(\bx)=0$.
In accordance with \cref{rad}, the corresponding radii are defined as follows:
\begin{align*}
\Rad{SR}{\ovclm}:=&
\inf
_{f\in \mathcal{F}_{\ovclm}}
\{\clm f(\bx)\mid
F+f \mbox{ is not subregular at }(\bx,\by)\},
\\
\Rad{SR}{lip+ss^*}:=&\inf_{f\in\mathcal{F}_{lip+ss^*}}\{\lip f(\bx) \mid
F+f \mbox{ is not subregular at } (\bx,\by)\},
\\
\Rad{SR}{\ovclm+ss^*}:=&
\inf
_{f\in \mathcal{F}_{\ovclm+ss^*}}
\{\clm f(\bx)\mid
F+f \mbox{ is not subregular at }(\bx,\by)\}.
\end{align*}
It is easy to check that
\begin{gather*}
\Rad{SR}{\ovclm}\leq \min\left\{\Rad{SR}{lip},\Rad{SR}{\ovclm+ss^*}\right\},\\
\max\left\{\Rad{SR}{lip},\Rad{SR}{\ovclm+ss^*}\right\}\le \Rad{SR}{lip+ss^*}.
\end{gather*}
%where $\Rad{SR}{lip}$ is defined by \cref{radSRlip}.

\begin{theorem}
%[Radius theorem]
\label{T3.2}
Let $X$ and $Y$ be Banach spaces,
$F:X\rightrightarrows Y$,
% be a mapping
%\red{with closed graph},
%metrically subregular at
and $(\bx,\by)\in\gph F$.
Then
\begin{gather}
\label{T3.2-1}
\Rad{SR}{lip}\leq{\rm srg}_{1}^+F(\bx,\by),
\\
\label{T3.2-2}
\Rad{SR}{\ovclm}\leq
{\rm srg}_{2}^+F(\bx,\by),
\\
\label{T3.2-3}
\Rad{SR}{\ovclm+ss^*}\leq
{\rm srg}_{4}^+F(\bx,\by).
\end{gather}
Suppose that $X$ and $Y$ are Asplund spaces and $\gph F$ is closed.
Then
\begin{gather}
\label{T3.2-4}
{\rm srg}_{1}F(\bx,\by)\le\Rad{SR}{lip}
\le{\rm srg}_{1}^+F(\bx,\by),
\\
\label{T3.2-5}
\Rad{SR}{\ovclm}
={\rm srg}_{2}F(\bx,\by),
\\
\label{T3.2-6}
\Rad{SR}{lip+ss^*}\ge{\rm srg}_{3}F(\bx,\by),
\\
\label{T3.2-7}
{\rm srg}_{4}F(\bx,\by)\leq\Rad{SR}{\ovclm+ss^*}
\leq{\rm srg}_{4}^+F(\bx,\by).
\end{gather}
\end{theorem}
\if{
\HG{20/11/21.
Unfortunately I do not know an upper bound for the radius $\Rad{SR}{lip+ss^*}$.}
}\fi
\if{
\AK{5/11/21 and 1/01/22.
The second
%inequality
part of \cref{T3.2}
is formulated in the Asplund space setting as a consequence of the sum rule in \cref{T2.2}(ii).
It should not be difficult to either find or prove the general Banach space version of \cref{T2.2}(ii) and as a consequence the second
%inequality
part of \cref{T3.2}
in terms of, e.g., Clarke normals, or just add the convex case.
}}\fi

\begin{remark}
\begin{enumerate}
\item
In view of \cref{P3.1}(vii) and (viii), in finite dimensions, estimates \cref{T3.2-4} in \cref{T3.2} recapture
%the first part of \cite[Theorem~3.2]{DonGfrKruOut20}.
\cref{T1.3-1} in \cref{T1.3}.
%\AK{25/04/22.
%It would be good to add a comment on the second part.}
\item
In view of \cref{P3.1}(iv), the upper bound ${\rm srg}_{1}^{+}F(\bx,\by)$ for $\Rad{SR}{lip}$ in \cref{T3.2-4} differs from the lower bound ${\rm srg}_{1}F(\bx,\by)$ by a factor of at most two.
As a consequence, in the Asplund space setting, the property of subregularity is stable under small Lipschitz continuous perturbations if and only if ${\rm srg}_{1}F(\bx,\by)>0$.
\item
Formula \cref{T3.2-5} gives the exact value for the radius of subregularity under firmly calm perturbations in Asplund spaces.
\item
We do not know an upper bound for the radius of subregularity under semismooth* and Lipschitz continuous perturbations.
\item
For firmly calm and semismooth* perturbations, we do not know how much the upper bound ${\rm srg}_{4}^{+}F(\bx,\by)$ differs from the lower bound ${\rm srg}_{4}F(\bx,\by)$.
\end{enumerate}
\end{remark}

The next lemma is
%the
a
key ingredient (together with \cref{T3.1}(i)) of the proof of the first part of \cref{T3.2}.
The proofs of the lemma and the second part of \cref{T3.2} are given in \cref{S4}.

\begin{lemma}
\label{L3.1}
Let $X$ and $Y$ be Banach spaces,
$F:X\rightrightarrows Y$,
$(\bx,\by)\in\gph F$, and $\ga>0$.
\begin{enumerate}
\item \label{L3.1(i)}
If ${\rm srg}_{1}^+F(\bx,\by)<\ga$, then there exists a function $f\in\mathcal{F}_{lip}$ such that
$\lip f(\bx)<\ga$ and
${\rm srg}_{1}^+(F+f)(\bx,\by)=0$.

\item \label{L3.1(ii)}
If ${\rm srg}_{2}^+F(\bx,\by)<\ga$, then there exists a function $f\in\mathcal{F}_{\ovclm}$ such that
$\clm f(\bx)<\ga$ and
${\rm srg}_{1}^+(F+f)(\bx,\by)=0$.

\item \label{L3.1(iii)}
If ${\rm srg}_{4}^+F(\bx,\by)<\ga$, then there exists a function {$f\in\mathcal{F}_{\ovclm+ss^*}$} such that
$\clm f(\bx)<\ga$ and
${\rm srg}_{1}^+(F+f)(\bx,\by)=0$.
\end{enumerate}
\end{lemma}

The next statement is a consequence of \cref{L3.1} and \cref{T3.1}(i).
It complements \cref{T3.1}(i) and immediately implies the estimates in the
%second
{first}
part of \cref{T3.2}.

\begin{corollary}
\label{C3.1}
Let $X$ and $Y$ be Banach spaces,
$F:X\rightrightarrows Y$,
$(\bx,\by)\in\gph F$, and $\ga>0$.
\begin{enumerate}
\item
If ${\rm srg}_{1}^+\,F(\bx,\by)<\ga$, then there exists a function $f\in\mathcal{F}_{lip}$ such that $\lip f(\bx)<\ga$, and $F+f$ is not
%metrically
subregular at $(\bx,\bar{y})$.
\item
If ${\rm srg}_{2}^+\,F(\bx,\by)<\ga$, then there exists a function $f\in\mathcal{F}_{\ovclm}$ such that $\clm f(\bx)<\ga$, and $F+f$ is not
%metrically
subregular at $(\bx,\bar{y})$.

\item
If ${\rm srg}_{4}^+\,F(\bx,\by)<\ga$, then there exists a function {$f\in\mathcal{F}_{\ovclm+ss^*}$} such that $\clm f(\bx)<\ga$, and $F+f$ is not
%metrically
subregular at $(\bx,\bar{y})$.
\end{enumerate}
\end{corollary}

\begin{example}
Let $F:\R\to\R$ be given by
\[F(x):=
\begin{cases}
x\sin\frac1x&\mbox{if $x\ne 0$,}\\
0&\mbox{if $x=0$,}
\end{cases}\]
and $\bx=\by=0$.
{The function $F$ is obviously subregular at $(0,0)$.
Next we show that the property is stable with respect to perturbations from $\mathcal{F}_{\ovclm+ss^*}$.}
For every $x\ne 0$, we have $y:=F(x)=x\sin\frac1x$ and $y'=\sin\frac1x -\frac1x\cos\frac1x$.
If $|y^*|=1$ and $x^*=y'y^*$, then
\begin{gather*}
\al:=\dfrac{|\ang{x^*,x-\bx}-\ang{y^*,y-\by}|} {\|(x^*,y^*)\|\|(x-\bx,y-\by)\|}=
\dfrac{\left|\left(\sin\frac1x -\frac1x\cos\frac1x\right)y^*x-y^*x\sin\frac1x\right|} {\max\{|x^*|,1\}\left(|x|+\left|x\sin\frac1x\right|\right)}
=
\dfrac{\left|\frac1x\cos\frac1x\right|} {\max\{|x^*|,1\}\left(1+\left|\sin\frac1x\right|\right)}.
\end{gather*}
If $|x^*|\to+\infty$, then $\left|\frac1x\cos\frac1x\right|\to+\infty$, in which case $\dfrac{\left|\frac1x\cos\frac1x\right|} {\max\{|x^*|,1\}}\to1$.
Hence, $\al\to0$ if and only if $\frac1x\cos\frac1x\to0$ as $x\to0$.
The latter condition obviously implies $\left|\sin\frac1x\right|\to1$, and consequently, $|x^*|\to1$ and $|y|/|x|\to1$.
By \cref{srg4,srg4+,Ups,Ups0},
%and \cref{P3.1},
${\rm srg}_{4}F(\bx,\by)=
{\rm srg}_{4}^+F(\bx,\by)=1$, and it follows from \cref{T3.2} that
$\Rad{SR}{\ovclm+ss^*}=1$.
\sloppy

At the same time, the property is not stable with respect to perturbations from $\mathcal{F}_{\ovclm}$.
Set $x_k:=(k\pi)^{-1}$ for all $k\in\N$.
Then $F(x_k)=0$, $F$ is differentiable at $x_k$, and $x_k\to\bx$.
By \cref{srg2,Del,Del0},
${\rm srg}_{2}F(\bx,\by)=0$, and it follows from \cref{T3.2} that
$\Rad{SR}{\ovclm}=0$.
%\magenta{By taking the sequence $x_k=(k\pi)^{-1}\to \bx$, we have $F(x_k)=0$. Since $F$ is Fr\'echet differentiable at $x_k$, we have $F'(x_k)\in D^*F(x_k)^(1)$ and we conclude that $\Rad{SR}{\ovclm}={\rm srg}_{2}^+F(\bx,\by)={\rm srg}_{2}F(\bx,\by)=0$.}
\end{example}

The formulas for $\Rad{SR}{\ovclm}$ in \cref{T3.2} together with definitions \cref{srg2,srg2+,Del,Del0}
suggest that the property of  subregularity is not very likely to be stable
%under calm and \ssstar perturbations
{with respect to perturbations from $\mathcal{F}_{\ovclm}$}
whenever the limit in \cref{ssrg} equals $0$, i.e., $F$ is not strongly subregular at $(\bx,\by)$.
The next example illustrates such an unlikely situation.

\begin{example}
Let $F:\R\to\R$ be given by
\[F(x):=
\begin{cases}
    [-x,x]&\mbox{if $x=1/k$ for some $k\in\N$},\\
    x&\mbox{otherwise},
\end{cases}\]
and $\bx=\by=0$.
With $x_k:=1/k$, we obviously have
$0\in F(x_k)$, and $\bx\ne x_k\to\bx$; hence, the limit in \cref{ssrg} equals $0$.
At the same time, for any $y\in(-x_k,x_k)$ and
$(x^*,y^*)\in N_{\gph F}(x_k,y)$, it holds $y^*=0$,
{i.e., $D^*F(x_k,y)(\Sp_{Y^*})=\es$}.
At all other points $(x,y)\in\gph F$ with $x\ne0$, i.e., when either $x=x_k$ and $y=\pm x_k$ for some $k\in\N$ or $y=x$, we have $\|y-\by\|/\|x-\bx\|=1$, and
{$D^*F(x,y)(\Sp_{Y^*})\ne\es$}.
By \cref{srg2,srg2+,Del,Del0},
%and \cref{P3.1},
${\rm srg}_{2}F(\bx,\by)=
{{\rm srg}_{2}^+F(\bx,\by)}
=1$, and it follows from \cref{T3.2} that
$\Rad{SR}{\ovclm}=1$.
\end{example}

In the case when the lower limit in \cref{ssrg} is strictly positive, i.e., $F$ is strongly subregular at $(\bx,\by)$, the property of subregularity is stable
{with respect to perturbations from $\mathcal{F}_{\ovclm}$} by \cref{T3.2}.
%Recall from \cref{ssrg} that this is actually the case of strong subregularity.
We now present the exact formula for the radii of strong subregularity
$\Rad{sSR}{\mathcal{P}}$ defined
in accordance with \cref{rad}
under calm, \extclm, or \extclm\ \ssstar perturbations, i.e.,
with $\mathcal{P}$ equal to $clm$, $\ovclm$, or $\ovclm+ss^*$.
\if{
Define:
\begin{align*}
\Rad{sSR}{\mathcal{P}}&:=
\inf_{f\in\mathcal{F}_{\mathcal{P}}}
\{\clm f(\bx)\mid
F+f \mbox{ is not strongly subregular at }(\bx,\by)\},
\if{
\Rad{sSR}{clm}&:=
\inf_{f\in\mathcal{F}_{clm}}
\{\clm f(\bx)\mid
F+f \mbox{ is not strongly subregular at }(\bx,\by)\},
\\
\Rad{sSR}{\ovclm}&:=
\inf_{f\in \mathcal{F}_{\ovclm}}
\{\clm f(\bx)\mid
F+f \\
&\hspace{35mm}\mbox{ is not strongly subregular at } (\bx,\by)\},
\\
\Rad{sSR}{\ovclm+ss^*}&:=
\inf_{f\in \mathcal{F}_{\ovclm+ss^*}}
\{\clm f(\bx)\mid
F+f \\
&\hspace{35mm}\mbox{ is not strongly subregular at } (\bx,\by)\}.
}\fi
\end{align*}
where $\mathcal{P}$ stands for $clm$, $\ovclm$, or $\ovclm+ss^*$.
}\fi

\begin{theorem}
\label{T3.3}
Let $X$ and $Y$ be Banach spaces, and $F:X\rightrightarrows Y$ be %a mapping
%\red{with closed graph},
{strongly subregular} at
$(\bx,\by)\in\gph F$.
Then
\begin{equation}
\label{T3.3-1}
\Rad{sSR}{\ovclm+ss^*}=\Rad{sSR}{\ovclm}=
\Rad{sSR}{clm}=\srg.
\end{equation}
\end{theorem}

\begin{remark}
The last equality in \cref{T3.3-1} strengthens the corresponding assertion in \cref{T1.2} (taken from \cite[Theorem~4.6]{DonRoc04}) which only guarantees
%the inequality \cref{T1.2-2}.
this equality in finite dimensions.
\end{remark}
\if{
\AK{20/04/22.
To be checked.}
\HG{25/05/22. I think it is ok.}
}\fi

The next lemma is
%the
a
key ingredient in the proof of \cref{T3.3}.
The proof of the lemma is given in \cref{S4}.

\begin{lemma}
\label{L3.2}
Let $X$ and $Y$ be Banach spaces, and $F:X\rightrightarrows Y$ be %a mapping
%\red{with closed graph},
{strongly subregular} at
$(\bx,\by)\in\gph F$.
Then
\begin{equation}
\label{Eq_ssrgLB}
\Rad{sSR}{{\ovclm}+ss^*}\leq \srg.
\end{equation}
\end{lemma}

\begin{proof}\emph{of \cref{T3.3}}
The statement is a consequence of \cref{L3.2} and the next chain of inequalities:
\begin{gather*}
\srg\leq\Rad{sSR}{clm}\leq
\Rad{sSR}{\ovclm}\leq
\Rad{sSR}{\ovclm+ss^*}.
\end{gather*}
For the first inequality we refer to \cite[Theorem~2.1]{CibDonKru18}, while the other two are trivial since $\mathcal{F}_{\ovclm+ss^*}\subset\mathcal{F}_{\ovclm} \subset\mathcal{F}_{clm}$.
\end{proof}
\if{
\AK{6/02/22.
Should \cref{Eq_ssrgLB} be formulated as a separate lemma with its proof moved to \cref{S4}?
}
\HG{11/02/22. Maybe this is a good idea.}
}\fi

\section{Proofs of the estimates for the radii of subregularity and strong subregularity}
\label{S4}

%\subsection{Proof of the upper estimate of
%$\Rad{SR}{lip}$}
%the radius of subregularity with respect to Lipschitz perturbations}
\subsection{Lipschitz perturbations: the upper estimate}
\label{S4.1}

We prove \cref{L3.1}\eqref{L3.1(i)} which provides
%the
a
key ingredient for the proof of the estimate \cref{T3.2-1} in \cref{T3.2}.

Suppose that ${\rm srg}_{1}^+F(\bx,\by)<\ga<+\infty$.
%Below we construct a
%function $f\in\mathcal{F}_{lip}$ such that
%$\lipf(\bx)<\ga$ and ${\rm srg}_{1}^+(F-f)(\bx,\by)=0$.
By definitions \cref{srg1+,Del}, there exist sequences %\magenta{$\gF\ni(x_k,y_k)\to(\bx,\by)$},
{$x_k\to\bx$ with $x_k\ne\bx$},
$\eps_k\downarrow0$,
{$y_k\in F(x_k)$},
$y_k^*\in\Sp_{Y^*}$, and
$x_k^*\in D^*_{\eps_k}F(x_k,y_k)(y_k^*)$
such that
\begin{align}
\label{T4.1P1}
\ga\,':={\sup_{k\in\N}
\left(\norm{x_k^*}+\frac{\|y_k-\by\|}{\|x_k-\bx\|}\right)<\ga.}
\end{align}
Denote $t_k:=\|x_k-\bx\|$.
By passing to subsequences, we can ensure that
\begin{gather}
\label{T4.1P2}
t_{k+1}<\frac{t_k}{2(k+1)}
\quad
(k\in\N).
\end{gather}
%for all $k\in\N$.
Set
\begin{align}
\label{T4.1P3}
\rho_k:=\frac{k}{k+1}t_k,
\end{align}
$\al_k:=t_k+\rho_k$ and $\be_k:=t_k-\rho_k$.
From \cref{T4.1P2} and \cref{T4.1P3} we obtain
\begin{align*}
\al_{k+1}=\frac{2k+3}{k+2}t_{k+1}< \frac{2k+3}{2(k+1)(k+2)}t_{k}<\frac{1}{k+1}t_{k}=\be_k.
\end{align*}
Observe that $\be_k<\|x-\bx\|<\al_k$ for all $x\in B_{\rho_k}(x_k)$.
Hence,
\begin{align}
\label{T4.1P4}
B_{\rho_k}(x_k)\cap B_{\rho_i}(x_i)=\es
\qdtx{for all}k\ne i.
\end{align}
For each $k\in\N$, choose a point $v_k\in Y$ such that \begin{align}
\label{T4.1P5}
\ang{y_k^*,v_k}=1\AND \|v_k\|<1+1/k.
\end{align}
For all $k\in\N$ and $x\in X$, set
\begin{subequations}\label{Eqf_k_LipPerm}
\begin{align}
%\label{T4.1P5}
s_{k}(x):=& \max\left\{1-\left(\|x-x_k\|/\rho_k\right)^{1+\frac1k}, 0\right\},
\\
%\label{T4.1P6}
g_{k}(x):=&y_k-\by+\ang{x_k^*,x-x_k}v_k,
\\
%\label{T4.1P7}
f_{k}(x):=&s_{k}(x)g_{k}(x).
\end{align}
\end{subequations}
Observe
%from \cref{T4.1P5} and \cref{T4.1P7}
that $s_{k}(x)=0$ and $f_{k}(x)=0$ for all $x\notin B_{\rho_k}(x_k)$.
In view of \cref{T4.1P4},
the function
\begin{equation}\label{Eqf_LipPerm}
f(x):=-\sum_{k=1}^\infty f_k(x),\quad x\in X,
\end{equation}
is well defined,
$f(x)=-f_k(x)$ for all $x\in B_{\rho_k}(x_k)$ and all $k\in\N$, and $f(x)=0$ for all $x\notin\cup_{k=1}^\infty B_{\rho_k}(x_k)$.
In particular, $f(\bx)=0$.
Observing that
$s_{k}(x_k)=1$, and
the function $s_{k}$ is differentiable at $x_k$ with
$\nabla s_{k}(x_k)=0$, we have
\begin{align}
\label{T4.1P13}
f(x_k)=\by-y_k\AND
\nabla f(x_k)=-\ang{x_k^*,\cdot}v_k
\qdtx{for all}k\in\N.
\end{align}
Given any $x,x'\in\overline{B}_{\rho_k}(x_k)$, by the mean-value theorem applied to the function $t\mapsto t^{1+\frac1k}$ on $\R_+$, there is a number $\theta\in[0,1]$ such that
\begin{gather*}
s_{k}(x)-s_{k}(x')= -\left(1+\frac1k\right)\rho_k^{-(1+\frac1k)} (\theta\|x-x_k\|
+(1-\theta)\|x'-x_k\|)^{\frac1k} (\|x-x_k\|-\|x'-x_k\|),
\end{gather*}
and consequently,
\begin{align*}
%\label{T4.1P8}
|s_{k}(x)-s_{k}(x')|\le \left(1+\frac1k\right)\rho_k^{-(1+\frac1k)} \max\{\|x-x_k\|,\|x'-x_k\|\}^{\frac1k}\|x-x'\|.
\end{align*}
%Thus, $s_{k}$ is Lipschitz continuous on $B_{\rho_k}(x_k)$ with modulus $\left(1+\frac1k\right)\rho_k^{-1}$.
%Further, in view of \cref{T4.1P6} and \cref{T4.1P7}, $f_k(x_k)=g_k(x_k)=y_k-\by$, and
%The function $f_{k}$ is differentiable at $x_k$ with
%$\nabla f_{k}(x_k)=\nabla g_{k}(x_k)=
%\ang{x_k^*,\cdot}v_k$.
If $x\ne x'$ and $\|x-x_k\|\le\|x'-x_k\|$, we obtain: \begin{align*}
\frac{|s_{k}(x)-s_{k}(x')|}{\|x-x'\|}\|g_k(x)\|
\le&\left(1+\frac1k\right) \frac{\|x'-x_k\|^{\frac1k}}{\rho_k^{1+\frac1k}} \big(\|y_k-\by\|+\|x_k^*\|\|v_k\|\|x-x_k\|\big)
\\
\le&\left(1+\frac1k\right)\left(\frac{\|y_k-\by\|}{\rho_k}+ \left(\frac{\|x'-x_k\|}{\rho_k}\right)^{1+\frac1k} \|x_k^*\|\|v_k\|\right),
\\
s_{k}(x')\frac{\|g_k(x)-g_k(x')\|}{\|x-x'\|}
\le&\left(1-\left(\frac{\|x'-x_k\|}{\rho_k}\right)^{1+\frac1k}\right) \|x_k^*\|\|v_k\|.
\end{align*}
Combining the above estimates, we get
\begin{align}
\notag
\frac{\|f_{k}(x)-f_{k}(x')\|}{\|x-x'\|}\le& \frac{|s_{k}(x)-s_{k}(x')|}{\|x-x'\|}\|g_k(x)\|+ s_{k}(x')\frac{\|g_k(x)-g_k(x')\|}{\|x-x'\|}
\\
\notag
\le&\left(1+\frac1k\right)\frac{\|y_k-\by\|}{\rho_k} +\frac1k\left(\frac{\|x'-x_k\|}{\rho_k}\right)^{1+\frac1k} \|x_k^*\|\|v_k\|
+\|x_k^*\|\|v_k\|
\\
\label{T4.1P11}
\le&\left(1+\frac1k\right)\left(\frac{\|y_k-\by\|}{\rho_k}+ \|x_k^*\|\|v_k\|\right),
\end{align}
and it follows from \cref{T4.1P1,T4.1P3,T4.1P5,T4.1P11} that
\begin{align*}
\frac{\|f_{k}(x)-f_{k}(x')\|}{\|x-x'\|}\le& \left(1+\frac1k\right)^2\left(\frac{\|y_k-\by\|}{t_k}+ \|x_k^*\|\right)\le\left(1+\frac1k\right)^2\ga\,'.
\end{align*}
Choose numbers $\hat k\in\N$ and $\ga\,''$ such that
$\left(1+1/\hat k\right)^2\ga\,'<\ga\,''<\ga$.
Thus, for any $k>\hat k$, the function $f_{k}$ is Lipschitz continuous on $\overline{B}_{\rho_k}(x_k)$ with modulus $\ga\,''$.
In particular,
given any $x\in{B}_{\rho_k}(x_k)$ with $x\ne x_k$, and taking $x':=x_k+\rho_k\frac{x-x_k}{\|x-x_k\|}\in\overline B_{\rho_k}(x_k)$, we have
$f_k(x')=0$, and consequently, $\|f_k(x)\|\le\ga\,''\|x-x'\|$.
Thus,
\begin{align}
\label{T4.1P12}
\|f_k(x)\|\le\ga\,''(\rho_k-\|x-x_k\|)
\qdtx{for all}x\in B_{\rho_k}(x_k).
\end{align}
Next, we show that the function $f$ is Lipschitz continuous with modulus $\ga\,''$ on $B_{\al_k}(\bx)$ for any $k>\hat k$.
Indeed, let $k>\hat k$, $x,x'\in B_{\al_k}(\bx)$ and $x\ne x'$.
\\
1) If $x,x'\in B_{\rho_i}(x_i)$ for some $i\ge k$, then, as shown above, $\|f(x)-f(x')\|<\ga\,''\|x-x'\|$.
\\
2) If $x\in B_{\rho_i}(x_i)$ and $x'\in B_{\rho_j}(x_j)$ for some $i,j\ge k$ with $i\ne j$, then, thanks to \cref{T4.1P4}, we have $\|x_i-x_j\|\ge\rho_i+\rho_j$, and using \cref{T4.1P12}, we obtain
\begin{align*}
%\label{T4.1P8}
\|f(x)-f(x')\|&\le\|f(x)\|+\|f(x')\|\le \ga\,''(\rho_i+\rho_j-\|x-x_i\|-\|x'-x_j\|)
\\&\le
\ga\,''(\|x_i-x_j\|-\|x-x_i\|-\|x'-x_j\|)
\le
\ga\,''\|x'-x'\|.
\end{align*}
3) If $x\in B_{\rho_i}(x_i)$ for some $i\ge k$, and $x'\notin\cup_{j=k}^\infty B_{\rho_j}(x_j)$, then $f(x')=0$, $\|x'-x_i\|\ge\rho_i$, and using \cref{T4.1P12}, we obtain
\begin{align*}
%\label{T4.1P8}
\|f(x)-f(x')\|
%=\|f(x)\|&\le\ga\,''(\rho_i-\|x-x_i\|)
\le
\ga\,''(\|x'-x_i\|-\|x-x_i\|)
\le
\ga\,''\|x'-x'\|.
\end{align*}
4) If $x,x'\notin\cup_{j=k}^\infty B_{\rho_j}(x_j)$, then $f(x)=f(x')=0$.
\\
Thus, in all cases,
$\|f(x)-f(x')\|\le\ga\,''\|x'-x'\|$.
Hence, $f\in{\mathcal F}_{lip}$ and $\lip f(\bx)\le\ga\,''<\ga$.
In view of \cref{T4.1P13,T4.1P5}, for all $k\in\N$, we have $\by=y_k+f(x_k)$ and
$D^*f(x_k)(y_k^*)=-\ang{y_k^*,v_k}x_k^*=-x_k^*$; hence $\by\in(F+f)(x_k)$ and,
%$0\in D^*_{\eps_k}F(x_k,y_k)(y_k^*)-D^*f(x_k)(y_k^*)$.
by \cref{T2.3},
$0\in D^*_{\eps_k'}(F+f)(x_k,\by)(y_k^*)$,
where
$\eps_k':=(\|\nabla f(x_k)\|+1)\eps_k$.
Thanks to \cref{T4.1P13,T4.1P5,T4.1P1},
$\eps_k'=(\|x_k^*\| \|v_k\|+1)\eps_k\le (2\|x_k^*\|+1)\eps_k<(2\ga+1)\eps_k$.
Thus, ${\rm srg}_{1}^+(F+f)(\bx,\by)=0$.
%\end{proof}

\subsection{
Firmly calm perturbations: the upper estimate}

We prove \cref{L3.1}\eqref{L3.1(ii)} which provides
%the
a
key ingredient for the proof of the estimate \cref{T3.2-2} in \cref{T3.2}.

Suppose that ${\rm srg}_{2}^+F(\bx,\by)<\ga<+\infty$.
%Below we construct a
%function $f\in\mathcal{F}_{lip}$ such that
%$\lipf(\bx)<\ga$ and ${\rm srg}_{1}^+(F-f)(\bx,\by)=0$.
By definitions \cref{srg2+,Del}, there exist sequences %\magenta{$\gF\ni(x_k,y_k)\to(\bx,\by)$},
{$x_k\to\bx$ with $x_k\ne\bx$},
$\eps_k\downarrow0$,
{$y_k\in F(x_k)$},
$y_k^*\in\Sp_{Y^*}$, and
$x_k^*\in D^*_{\eps_k}F(x_k,y_k)(y_k^*)$
such that
\begin{gather}
\label{6.2-0}
\lim_{k\to\infty}\eps_k\norm{x_k^*}=0,
\\
\label{6.2-1}
\ga\,':=\sup_{k\in\N}
\frac{\|y_k-\by\|}{\|x_k-\bx\|}<\ga.
\end{gather}
Denote $t_k:=\|x_k-\bx\|$.
By passing to subsequences, we can ensure
estimate \cref{T4.1P2}.
Set
\begin{align}
\label{6.2-2}
\rho_k:=\min\left\{\frac{1}{k+1},
\frac{\ga\,'}{
{(k+1)}
(1+\|x_k^*\|)}\right\}t_k,
\end{align}
$\al_k:=t_k+\rho_k$ and $\be_k:=t_k-\rho_k$,
and observe that, thanks to \cref{T4.1P2}, that
\begin{align*}
\be_k-\al_{k+1}\ge
\frac{k t_k}{k+1}-\frac{k+3}{k+2}t_{k+1}>&
\left(k-\frac{k+3}{2(k+2)}\right)\frac{t_k}{k+1}
>
\frac{(k-1)t_k}{k+1}\ge0.
\end{align*}
%Note that condition \cref{T4.1P4} holds because
%$\rho_k\leq kt_k/t_{k+1}$.
As a consequence, condition \cref{T4.1P4} holds.
%As in the proof of \cref{L3.1}\eqref{L3.1(i)} in \cref{S4.1},
For each $k\in\N$, we choose a point $v_k\in Y$
%such that $\ang{y_k^*,v_k}=1$ and
%{$\|v_k\|<1+\frac1k$},
satisfying \cref{T4.1P5},
and then define functions $s_k, g_k, f_k$, and $f$ by \cref{Eqf_k_LipPerm,Eqf_LipPerm}.
Thanks to \cref{T4.1P4}, function $f$ is well defined.
Moreover, it is differentiable at $x_k$, and satisfies \cref{T4.1P13}.
Hence $\by\in(F+f)(x_k)$ and,
%$0\in D^*_{\eps_k}F(x_k,y_k)(y_k^*)-D^*f(x_k)(y_k^*)$.
by \cref{T2.3},
$0\in D^*_{\eps_k'}(F+f)(x_k,\by)(y_k^*)$,
where
$\eps_k':=(\|\nabla f(x_k)\|+1)\eps_k$.
Thanks to \cref{T4.1P13,T4.1P5,6.2-0},
$\eps_k'=(\|x_k^*\| \|v_k\|+1)\eps_k\le (2\|x_k^*\|+1)\eps_k\to0$ as $k\to\infty$.
Thus, ${\rm srg}_{1}^+(F+f)(\bx,\by)=0$.
%\HG{11/02/22. Here we need existence of $x_k^*, y_k^*$ and $\|x_k^*\|\eps_k\to 0$}

As shown in \cref{S4.1}, estimates \cref{T4.1P11} hold true for all $x,x'\in\overline{B}_{\rho_k}(x_k)$, i.e., $f$ is Lipschitz continuous on $\overline{B}_{\rho_k}(x_k)$ with modulus $l_k:=(1+1/k)({\|y_k-\by\|}/{\rho_k}+ \|x_k^*\|\|v_k\|)$.
Since, $f(x)=0$ for all $x\notin\cup_{k=1}^\infty B_{\rho_k}(x_k)$, the function $f$ is Lipschitz continuous around every $x\ne\bx$ near $\bx$.
We
%shall
now
%prove
show
that $\clm f(\bx)<\ga$.
Choose numbers $\hat k\in\N$ and $\ga\,''$ such that
$\left(1+1/\hat k\right)^2\ga\,'<\ga\,''<\ga$.
Consider a point $x\in X$ with
$0<\|x-\bx\|<t_{\hat k}$ and
$f(x)\ne0$.
Then there is a unique index $k=k(x)\ge\hat k$ such that
$x\in B_{\rho_k}(x_k)$ and $f(x)=-f_k(x)$.
Thus, by \cref{6.2-2,T4.1P5}, we have
\begin{gather*}
\|x-\bx\|>t_k-\rho_k\geq k t_k/(k+1),\\
\|g_k(x)\|\leq \|y_k-\by\|+\rho_k\|x_k^*\|\|v_k\|< \|y_k-\by\|+\ga\,'t_k/{k}.
%\red{\left(1+\frac1k\right)}.
\end{gather*}
Since $s_k(x)\leq1$ and $s_k(\bx)=0$,
we obtain from \cref{Eqf_k_LipPerm,6.2-1}:
\[\frac{\|f(x)-f(\bx)\|}{\|x-\bx\|}\le\frac{\|g_k(x)\|}{\|x-\bx\|}< \left(1+\frac1k\right)\left(\frac{\|y_k-\by\|}{\|x_k-\bx\|}+ \frac{\ga\,'}{k}\right)\le\left(1+\frac1k\right)^2 \ga\,'<\ga\,''.\]
Hence, $\clm f(\bx)\le\ga\,''<\ga$.

%\section{Proof of the upper estimate of
%$\Rad{SR}{clm+ss^*}$}
%the radius of subregularity with respect to calm semismooth* perturbations}
\subsection{
{Firmly calm} semismooth* perturbations: the upper estimate}
\label{S4.3}

%Now let us consider upper bounds.

We prove \cref{L3.1}\eqref{L3.1(iii)} which provides
%the
a
key ingredient for the proof of the estimate \cref{T3.2-3} in \cref{T3.2}.

Suppose that ${\rm srg}_{4}^+F(\bx,\by)<\ga<+\infty$.
By definitions \cref{srg4+,Ups}, there exist
sequences
$x_k\to\bx$ with $x_k\ne\bx$,
$\eps_k\downarrow0$,
{$y_k\in F(x_k)$},
$y_k^*\in\Sp_{Y^*}$, and
$x_k^*\in D^*_{\eps_k}F(x_k,y_k)(y_k^*)$
such that conditions \cref{6.2-0,6.2-1} are satisfied, and
\begin{align}
\label{T6.2P1}
\dfrac{|\ang{x_k^*,x_k-\bx}-\ang{y_k^*,y_k-\by}|} {\|x_k^*,y_k^*\|\|(x_k-\bx,y_k-\by)\|}\leq\eps_k.
\end{align}
For every $k\in\N$,
set $t_k:=\norm{x_k-\bx}$ and $u_k:=(x_k-\bx)/t_k\in\Sp_X$, and
choose points $u_k^*\in\Sp_{X^*}$ and $v_k\in Y$ such that \begin{align}
\label{T6.2P3}
\ang{u_k^*,u_k}=
%\|u_k^*\|=
1,\quad
\ang{y_k^*,v_k}=1,\quad
\norm{v_k}<2.
\end{align}
Set
\begin{align}
\label{T6.2P6}
\alpha_k(x):=\ang{u_k^*,x},\quad
r_k(x):=x-\alpha_k(x)u_k,\quad
\hat x_k^*:=x_k^*-\ang{x_k^*,u_k}u_k^*.
\end{align}
Observe that
\begin{gather}
\label{T6.2P7}
\al_k(u_k)=1,
\quad
r_k(u_k)=0,
\quad
\ang{\hat x_k^*,u_k}=0,
\quad
\ang{\hat x_k^*,r_k(x)}=\ang{\hat x_k^*,x}.
\end{gather}

{\em Case 1: no element of the sequence $(u_k)$ appears infinitely many times.}
Without changing the notation, we now construct inductively a subsequence of $(u_k)$ (and the corresponding to it other involved sequences).
% with special properties.
For each $k=1,\ldots$, we do the following.
If $u_k=\lim_{j\to\infty}u_j$, remove $u_{k}$ and all its finitely many copies from the sequence.
Then $s_{k}:=\limsup_{j\to\infty}\|u_{j}-u_k\|>0$.
Choose a subsequence of $(u_j)_{j=k+1}^{\infty}$ such that
$\|u_{j}-u_k\|>s_k/2$ for all $j>k.$
\if{
By construction,
\[\rho_k:=\inf\{\|u_j-u_k\|\mid j\ne k\}\geq\min_{j\le k}s_j/2.\]
Hence, the sequence $(\rho_k)$ is nonincreasing, and
$\|u_k-u_i\|\geq \max\{\rho_k,\rho_i\}$ for all $k\ne i$.
}\fi
Defining
\[\rho_k:=\inf\{\|u_j-u_k\|\mid j\ne k\}/2,\]
we have $\rho_k\geq \min_{j\le k}s_j>0$ for all $k\in\N$, and $\|u_k-u_i\|\geq2\max\{\rho_k,\rho_i\}$ for all $k\ne i$,
and consequently,
%condition \cref{T4.1P4} is satisfied.
\begin{align}
\label{T6.2P5}
B_{\rho_k}(u_k)\cap B_{\rho_i}(u_i)=\es
\qdtx{for all}k\ne i.
\end{align}

{Let $\ga''\in(\ga',\ga)$.}
For each $k\in\N$, choose a number $\tau_k\in(0,\rho_k/2)$ such that $\tau_k{\|x_k^*\|}<(\ga''-\ga\,')/4$, and
define
\begin{align}
\label{T6.2P8-1}
s_k(x)&:=
\begin{cases}
\max\left\{1-\left(\dfrac{\norm{r_k(x)}} {\tau_k\alpha_k(x)}\right)^2,0\right\}&\mbox{if $\alpha_k(x)>0$},\\
   0&\mbox{otherwise},
\end{cases}
\\
\label{T6.2P8-2}
A_k(x)&:=
\alpha_k(x)(y_k-\by)/t_k+\ang{\hat x^*_k,x}v_k,
\\
\label{T6.2P8-3}
\hat A_k(x)&:=s_k(x)A_k(x).
\end{align}
Observe that $s_k(x)=0$ and $\hat A_k(x)=0$ for all $x\in X$ with $\norm{r_k(x)}\ge\tau_k\alpha_k(x)$.
If $\norm{r_k(x)}<\tau_k\alpha_k(x)$, then
by \cref{T6.2P6}, $x\ne0$ and
%, thanks to \cref{T6.2P3} and \cref{T6.2P6}, we have
\begin{align*}
\big\|x-\norm{x}u_k\big\|= \norm{r_k(x)+(\alpha_k(x)-\norm{x})u_k}\le& \|r_k(x)\|+\norm{x}-\alpha_k(x)
\\\le&
2\|r_k(x)\|<2\tau_k\alpha_k(x)<\rho_k\|x\|;
\end{align*}
hence, $x/\norm{x}\in B_{\rho_k}(u_k)$.
It follows from \cref{T6.2P5} that
for every $x\in X$ there is at most one $k\in\N$ with
$\hat A_k(x)\ne0$.
Thus, the function
\begin{align}
\label{T6.2P9}
f(x):=-\sum_{k=1}^\infty \hat A_k(x-\bx),\quad x\in X,
\end{align}
is well defined.
The function $\hat A_k(x)$ is Lipschitz continuous around any point $x\ne0$; hence, $f$ is Lipschitz continuous around any point $x\ne\bx$.
Further, $s_k$ is positively homogenous at $0$, while
$A_k$ is
%continuous
linear.
Therefore $\hat A_k$ is positively homogenous at $0$, and $f$
is positively homogenous at $\bx$ and, by \cref{CorSS}, \ssstar at $\bx$.
If $\hat A_k(x)\ne0$, then $\norm{r_k(x)}<\tau_k\alpha_k(x)$ and,
thanks to \cref{T6.2P3,T6.2P7,T6.2P6,6.2-1},
\begin{align*}
\notag
\|\hat A_k(x)\|\le&\|A_k(x)\|\le\alpha_k(x)\norm{y_k-\by}/t_k+ 2|\ang{\hat x^*_k,x}|\leq
\alpha_k(x)\ga'+2|\ang{\hat x^*_k,r_k(x)}|
\\\le&
%\label{T6.2P9}
(\ga'+2\norm{\hat x^*_k}\tau_k)\alpha_k(x)\le
(\ga'+4\norm{x^*_k}\tau_k)\|x\|<
(\ga'+(\ga''-\ga'))\|x\|=\ga''\|x\|,
\end{align*}
and consequently,
$\clm\hat A_k(0)\leq\ga''$.
Hence, $\clm f(\bx)\leq\ga''<\ga$.
Thus, $f\in {\mathcal F}_{{\ovclm}+ss^*}$.

In view of \cref{T6.2P8-1,T6.2P7},
$s_k(x_k-\bx)=1$.
Further,
$s_k$ is Fr\'echet differentiable at $x_k-\bx$ with derivative
$\nabla s_k(x_k-\bx)=0$.
Hence, using \cref{T6.2P7} again, we obtain:
\begin{gather}
\label{T6.2P15}
f(x_k)=-\hat A_k(x_k-\bx)=\by-y_k,
\\
\label{T6.2P16}
\nabla f(x_k)=-\nabla \hat A_k(x_k-\bx)= -\nabla A_k(x_k-\bx)= -\ang{u_k^*,\cdot}(y_k-\by)/t_k-
\ang{\hat x_k^*,\cdot}v_k.
\end{gather}
Thanks to \cref{T6.2P3} and \cref{T6.2P6}, the latter equality yields
\begin{align*}
%\label{T6.2P17}
D^*f(x_k)(y_k^*)=
-\ang{y_k^*,y_k-\by}u_k^*/t_k-
\hat x_k^*\ang{y_k^*,v_k}=
-x_k^*-(\ang{y_k^*,y_k-\by}-\ang{x_k^*,x_k-\bx})u_k^*/t_k.
\end{align*}
By \cref{T6.2P15},
%and \cref{T6.2P17},
$\by\in(F+f)(x_k)$, and, by \cref{T2.3},
\begin{align*}
%\label{T6.2P18}
\hat u_k^*:= -(\ang{x_k^*,x_k-\bx}-\ang{y_k^*,y_k-\by})u_k^*/t_k\in D^*_{\eps_k'}(F+f)(x_k,\by)(y_k^*),
\end{align*}
where $\eps_k':=(\|\nabla f(x_k)\|+1)\eps_k$.
In view of \cref{6.2-1,T6.2P1,T6.2P6,T6.2P16}, we have
\begin{align}
\notag
\|\hat u_k^*\|&= |\ang{x_k^*,x_k-\bx}-\ang{y_k^*,y_k-\by}|/t_k
\\&=
\notag
\dfrac{|\ang{x_k^*,x_k-\bx}-\ang{y_k^*,y_k-\by}|} {\|(x_k^*,y_k^*)\|
\|(x_k-\bx,y_k-\by)\|}\cdot
\frac{\norm{x_k-\bx}+\norm{y_k-\by}}{t_k}\cdot
\max\{\|x_k^*\|,1\}
\\&\le
\label{T6.2P17}
\eps_k(1+\gamma)\max\{\|x_k^*\|,1\},
\\
\label{T6.2P17-2}
\eps_k'&<(\ga+4{\|x_k^*\|}+1)\eps_k.
\end{align}
It follows from \cref{6.2-0} that $\hat u_k^*\to0$ and $\eps_k'\to0$ as $k\to\infty$.
Thus, ${\rm srg}_{1}^+(F+f)(\bx,\by)=0$.

{\em Case 2: There is an element $u$ of the sequence $(u_k)$ which appears infinitely many times.}
From now on, we consider the stationary subsequence $(u)$ of the sequence $(u_k)$, i.e. without changing the notation, we assume that $(x_k-\bx)/t_k=u\in\Sp_{X}$ for all $k\in\N$.
Thus, subscript $k$ can be dropped in many of the formulas in \cref{T6.2P3,T6.2P6,T6.2P7}:
\begin{gather}
\label{T6.2P20}
\ang{u^*,u}
=\|u^*\|
=1,\;\;
\alpha(x):=\ang{u^*,x},\;\;
r(x):=x-\alpha(x)u,
\;\;
\hat x_k^*:=x_k^*-\ang{x_k^*,u}u^*,
\\
\label{T6.2P21}
\al(u)=1,
\quad
r(u)=0,
\quad
\ang{\hat x_k^*,u}=0,
\quad
\ang{\hat x_k^*,r(x)}=\ang{\hat x_k^*,x}.
%|\ang{\hat x_k^*,x}|\le\|\hat x_k^*\|\|r(x)\|,
%\;\;
\end{gather}
By passing to a subsequence again, we can also assume that the sequence $(t_k)$ satisfies
\begin{gather}
\label{T6.2P22}
t_{k}<e^{-k}t_{k-1},
\qdtx{or equivalently,}\ln(t_{k-1}/t_k)>k
\quad(k\ge2).
\end{gather}
{Let $\ga''\in(\ga',\ga)$.}
For each $k\in\N$, choose a number $\tau_k\in(0,1)$ such that $\tau_k{\|x_k^*\|}<(\ga''-\ga\,')/4$, and consider
mappings $s_k$, $A_k$ and $\hat A_k:X\to Y$ ($k\ge1$) defined by \cref{T6.2P8-1,T6.2P8-2,T6.2P8-3}.
Note that in the considered case $\al_k$ and $r_k$ do not depend on $k$; hence, the first two mappings take a slightly simplified form:
\begin{align}
\label{T6.2P23.s_k}
s_k(x):=&
\begin{cases}
\max\left\{1-\left(\dfrac{\norm{r(x)}} {\tau_k\alpha(x)}\right)^2,0\right\}&\mbox{if $\alpha(x)>0$},\\
   0&\mbox{otherwise},
\end{cases}
\\
\label{T6.2P23}
A_k(x):=&\alpha(x)(y_k-\by)/t_k+
\ang{\hat x^*_k,x}v_k.
\end{align}
Next we define mappings $T_k:\{x\in X\mid \alpha(x)>0\}\to Y$ ($k\ge2$) and $g,f:X\to Y$:
\begin{align}
\label{T6.2P23.2}
T_k(x):=&\hat A_k(x)+\frac{\ln(\alpha(x)/t_{k})} {\ln(t_{k-1}/t_k)}(\hat A_{k-1}(x)-\hat A_k(x)),
\\
\label{T6.2P31}
g(x):=&
\begin{cases}
\hat A_1(x)&\mbox{if $\alpha(x)\geq t_1$,}\\
T_k(x)&\mbox{if $\alpha(x)\in[t_k,t_{k-1})$ for some $k\geq 2$,}\\
0&\mbox{if $\alpha(x)\leq 0$},
\end{cases}
\\
\label{T6.2P32}
f(x):=&-g(x-\bx).
\end{align}
We will now show that $g$ is Lipschitz continuous around every point in $U:=\{x\in X\mid 0<\norm{x}< t_1\}$, and \ssstar and calm at $0$ with $\clm g(0)\leq\ga''$.

The mappings $T_k$ are Lipschitz continuous around every
%$x\not=0$
$x\in X$ with $\al(x)>0$
as compositions of mappings having this property, and satisfy
\begin{gather}
\label{T6.2P24}
T_k(x)=T_{k+1}(x)=\hat A_k(x)\qdtx{whenever}\alpha(x)=t_k.
\end{gather}
Hence, $g$ is Lipschitz continuous around every point $x\in U$ with $\al(x)>0$.
By \cref{T6.2P23.s_k},
%for every $k\in\N$,
the functions $s_k$ vanish on the open set $U_0:=\{x\mid\alpha(x)<\|r(x)\|\}$.
By \cref{T6.2P31,T6.2P23.2,T6.2P8-3},
so does $g$.
%Hence, $g$ is Lipschitz continuous around every point in $U_0$.
If $\alpha(x)=\|r(x)\|=0$, then, by \cref{T6.2P20}, $x=0$.
Hence, $\al(x)>0$ for all $x\in U\setminus U_0$, and consequently, $g$ is Lipschitz continuous around every point in $U$.

As shown in the preceding case,
%it is easy to see that
$\hat A_k$ is positively homogeneous at $0$, and
\begin{align}\label{EqClmAk}
\|\hat A_k(x)\|< \ga''\| x\|\qdtx{for all} x\ne0.
\end{align}
Let $x\in U\setminus U_0$.
Then $0<\al(x)<t_1$,
and we can find a unique integer $\hat k\ge2$ such that $\alpha(x)\in[t_{\hat k},t_{\hat k-1})$.
Thus, $\xi:=\ln(\alpha(x)/t_{\hat k})/\ln(t_{\hat k-1}/t_{\hat k})\in[0,1)$, and from \Cref{EqClmAk,T6.2P31,T6.2P23.2} we obtain:
\[\|g(x)\|=\|(1-\xi)\hat A_{\hat k}(x)+\xi \hat A_{\hat k-1}(x)\|<(1-\xi)\ga''\|x\|+\xi\ga''\|x\|=\ga''\|x\|,\]
showing $\clm g(0)\leq \ga''$.

The mapping $A_k$ is linear and continuous, while the function $s_k$, in view of \cref{T6.2P23.s_k,T6.2P20}, satisfies $s_k(x\pm tx)=s_k(x)$ for all $t\in(0,1)$.
Hence, $\hat A_k'(x;\pm x)=\pm\hat A_k(x)$.
It is easy to check from \cref{T6.2P20} that $\big(\ln(\al(\cdot)/t_k)\big)'(x;\pm x)=\pm1$.
Now it follows from \cref{T6.2P23.2} that
\begin{gather}
\label{T6.2P24.T_k'}
T_k'(x;\pm x)=\pm\left(T_k(x)+\frac{\hat A_{k-1}(x)-\hat A_{k}(x)}{\ln(t_{k-1}/t_{k})}\right).
\end{gather}
If $\alpha(x)\in(t_{\hat k},t_{\hat k-1})$, then, by \cref{T6.2P31}, $g$ coincides with $T_{\hat k}$ around $x$, yielding $g'(x;\pm x)=T_{\hat k}'(x;\pm x)$.
Taking into account \cref{T6.2P24.T_k',T6.2P22,EqClmAk}, we obtain:
\begin{gather}
\label{T6.2P25}
\|g(x)-g'(x;x)\| =
\frac{\|\hat A_{\hat k-1}(x)-\hat A_{k}(\hat x)\|}
{\ln(t_{\hat k-1}/t_{\hat k})}<\frac{2\ga\|x\|}{\hat k},
\end{gather}
and similarly, $\|g(x)+g'(x;-x)\|<2\ga\|x\|/\hat k$.
If $\alpha(x)=t_{\hat k}$ $(>0)$, then we have $g'(x,x)=T_{\hat k}(x;x)$, yielding \cref{T6.2P25}, and $g'(x;-x)=T'_{\hat k+1}(x;-x)$, yielding
\[ \|g(x)+ g'(x;- x)\| =\frac{\|\hat A_{\hat k}(x)-\hat A_{\hat k+1}(x)\|}{\ln(t_{\hat k}/t_{\hat k+1})}<\frac{2\ga\|x\|}{\hat k+1}.\]
Since $\alpha(x)\to 0$ and $\hat k\to\infty$ as
$U\setminus U_0\ni x\to0$, it follows from \cref{LemSS} that $g$ is \ssstar at $0$.
We conclude that the function $f$ defined by \cref{T6.2P32} belongs to
${\mathcal F}_{{\ovclm}+ss^*}$, and
$\clm f(\bx)\leq\ga''<\ga$.

Now, consider the functions $h_k:=g-\hat A_k$ ($k\geq 2$).
We claim that
$\clm h_k(t_ku)<c_k:=2\ga/k$.
By \cref{T6.2P21,T6.2P31,T6.2P24}, $\al(t_ku)=t_k$, $g(t_ku)=T_k(t_ku)=\hat A_k(t_ku)$, and consequently, $h_k(t_ku)=0$.
In view of \cref{T6.2P31,T6.2P23.2}, the function $h_k$ admits the following representation near $t_ku$:
\[h_k(x)=
\begin{cases}h_k^1(x)&\mbox{if $\alpha(x)\in[t_k,t_{k-1})$},\\
h_k^2(x)&\mbox{if $\alpha(x)\in[t_{k+1},t_k)$,}
\end{cases}\]
where
\[h_k^1(x):=\frac{\ln(\alpha(x)/t_k)}{\ln(t_{k-1}/t_k)} (\hat A_{k-1}(x)-\hat A_k(x)),\quad
h_k^2(x):=\frac{\ln(\alpha(x)/t_k)}{\ln(t_{k+1}/t_k)} (\hat A_{k+1}(x)-\hat A_k(x)).\]
%Hence, in order to verify the claim it suffices to show that both mappings $h_k^i$, $i=1,2$ are calm at $t_ku$ and $\clm h_k^i(t_k u)<2\gamma/k$.
By \cref{T6.2P21,T6.2P23.s_k},
%$r(t_ku)=0$ implying that
$s_k(t_ku)=1$, and $s_k$ is Fr\'echet differentiable at $t_ku$ with $\nabla s_k(t_ku)=0$.
Moreover, $\ln(\alpha(t_ku)/t_k)=0$.
Thus, in view of \cref{T6.2P20,T6.2P21,T6.2P8-3}, $h_k^1$ is  differentiable at $t_k u$, and
%by taking into account $\alpha(t_ku)=t_k$, $(r(t_ku)=0$,
its derivative amounts to
\[\nabla h_k^1(t_ku)= \frac{\alpha(\cdot)}{\ln(t_{k-1}/t_k)} (A_{k-1}(u)-A_k(u))=\frac{\alpha(\cdot)}{\ln(t_{k-1}/t_k)} \left(\frac{y_{k-1}-\by}{t_{k-1}}-\frac{y_k-\by}{t_k}\right).\]
Hence, in view of \cref{6.2-1,T6.2P22}, we have
$\|\nabla h_k^1(t_ku)\|<2\ga/k=c_k.$
The same arguments yield
$\|\nabla h_k^2(t_ku)\|<2\ga/(k+1)<c_k$, and consequently,
\begin{align}
\label{T6.2P40}
\clm h_k(t_ku)=
\max\{\|\nabla h_k^1(t_ku)\|,
\|\nabla h_k^2(t_ku)\|\}<c_k.
\end{align}

Next, for all $x\in X$ and $k\in\N$, set
${C}_k(x):=\hat A_k(x-\bx)$.
Then
\begin{align}
\label{T6.2P41}
{C}_k(x_k)=s(t_ku)A_k(t_ku)=t_kA_k(u)=y_k-\by.
\end{align}
and  ${C}_k$ is Fr\'echet differentiable at $x_k$ with
\begin{align}
\label{T6.2P42}
\nabla {C}_k(x_k)=\nabla A_k(t_ku)= \ang{u^*,\cdot}(y_k-\by)/t_k+\ang{\hat x_k^*,\cdot}v_k.
\end{align}
Thus, in view of \cref{T6.2P20},
\[D^*{C}_k(x_k)(y_k^*)=\ang{y_k^*,y_k-\by}u^*/t_k+\hat x_k^*= x_k^*-\hat u_k^*.\]
where
$\hat u_k^*:=\left(\ang{x_k^*,u}-\ang{y_k^*,y_k-\by}\right)u^*/t_k$.
By \cref{T6.2P41},
$\by\in(F-{C}_k)(x_k)$ and, by \cref{T2.3},
\begin{align}
\label{T6.2P44}
\hat u_k^*\in D^*_{\eps_k'}(F-{C}_k)(x_k,\by)(y_k^*),
\end{align}
where $\eps_k':=(\|\nabla {C}_k(x_k)\|+1)\eps_k$.
In view of \cref{6.2-1,T6.2P42,T6.2P1,T6.2P20}, estimates \cref{T6.2P17,T6.2P17-2} hold true.
It follows from \cref{6.2-0} that $\hat u_k^*\to0$ and $\eps_k'\to0$ as $k\to\infty$.
%Thus, ${\rm srg}_{1}^+(F+f)(\bx,\by)=0$.
Observe that $f(x)=-{C}_k(x)-h_k(x-\bx)$ for all
%$x$ belonging to some neighborhood of $x_k$,
{$x\in X$}.
By \cref{T6.2P40,T6.2P41}, we have
\begin{align}
\label{T6.2P46}
f(x_k)=-{C}_k(x_k)=\by-y_k\AND
\clm (f+{C}_k)(x_k)=
\clm h_k(t_ku)<c_k.
\end{align}
In view of \cref{T6.2P44,T6.2P46},
it follows from \cref{LemEpsCoder} that
$\hat u_k^*\in D^*_{\de_k}(F+f)(x_k,\by)(y_k^*)$,
where $\de_k:=(\eps_k'+c_k)/(1-c_k)\to0$ as $k\to\infty$.
Hence,
{${\rm srg}_{1}^+(F+f)(\bx,\by)=0$}.
%\end{proof}

\subsection{The lower estimates}

We prove the lower estimates in \cref{T3.2-4,T3.2-6,T3.2-5,T3.2-7}.
Combined with the first part of \cref{T3.2}, they prove the second part of the theorem.

Suppose that $F+f$ is not
%metrically
subregular at $(\bx,\by)$ {for some function $f\in{\mathcal F}_{\ovclm}$}.
It follows from \cref{T3.1}(ii)
that
${\rm srg}_{1}(F+f)(\bx,\by)=0$.
By definitions \cref{srg1,Del}, there exist sequences
{$x_k\not=\bx$}, $y_k\in F(x_k)$,
$y^*_k\in\Sp_{Y^*}$ and
$x_k^*\in D^*(F+f)(x_k,y_k+f(x_k))(y_k^*)$
such that
{$f$ is Lipschitz continuous near $x_k$, and}
\begin{gather}
\label{Tr1}
\al_k:=\max\left\{\|x_k-\bx\|,
%\magenta{\|y_k-\by\|,}
\|x_k^*\|, \frac{\|y_k+f(x_k)-\by\|}{\|x_k-\bx\|}\right\}\to0.
\end{gather}
For each $k\in\N$, set
$t_k:=\|x_k-\bx\|$, and choose a positive number
\begin{gather}
\label{Tr2}
\eps_k<\min\left\{\frac{\al_kt_k}{\al_k+1},\frac12\right\}.
\end{gather}
Then $0<\eps_k<t_k\le\al_k$.
By \cref{T2.2}(ii),
there exist
{$x_{1k},x_{2k}\in B_{\eps_k}(x_k)$},
$y_{1k}\in F(x_{1k})\cap B_{\eps_k}(y_k)$,
$y_{1k}^*,y_{2k}^*\in B_{\eps_k}(y_k^*)$,
$x_{1k}^*\in D^*F(x_{1k},y_{1k})(y_{1k}^*)$ and
$x_{2k}^*\in D^*f(x_{2k})(y_{2k}^*)$ such that
\begin{gather}
\label{Tr3}
\|x_{1k}^*+x_{2k}^*-x_k^*\|<\eps_k.
\end{gather}
Observe that
\begin{gather}
\label{Tr4}
\min\{\|y_{1k}^*\|,\|y_{2k}^*\|\}>1-\eps_k>\frac12,
\\
\label{Tr5}
\min\{\norm{x_{1k}-\bx},\norm{x_{2k}-\bx}\}> t_k-\eps_k> \frac{t_k}{\al_k+1}{>0}.
\end{gather}
\if{
\HG{29/01/2022. Here we need that $f$ belongs to the class $\mathcal{F}_{\ovclm+}$
Otherwise we cannot apply \cref{T2.2}(ii).}
\AK{8/02/22.
I see.
Is `+' in `$\ovclm+$' a typo?}
}\fi
Let $\clm f(\bx)<\ga<+\infty$, and
choose a number $\ga\,'\in(\clm f(\bx),\ga)$.
In view of \cref{Tr1,Tr2,Tr5},
we have for all sufficiently large $k\in\N$:
\begin{align}
\notag
\frac{\|y_{1k}-\by\|}{\|x_{1k}-\bx\|}&<
\frac{\|y_{k}-\by\|+\eps_k}{(\al_k+1)\iv t_k}\le
\frac{\|f(x_k)\|+\|y_k+f(x_k)-\by\|+\eps_k} {(\al_k+1)\iv t_k}
\\&
\label{Tr6}
<(\al_k+1)\left(\ga\,'+\al_k\right)+\al_k= \ga\,'+\al_k(\ga\,'+\al_k+1)<\ga.
\end{align}
In particular, $y_{1k}\to\by$ as $k\to\infty$.
Letting $\ga\downarrow\clm f(\bx)$ and observing that the function {$f\in\mathcal{F}_{\ovclm}$} is arbitrary,
inequality \cref{Tr6} implies
$\Rad{SR}{\ovclm}
\ge{\rm srg}_{2}F(\bx,\by)$.
Combined with the established above inequality \cref{T3.2-2} and \cref{P3.1}(v), this
proves \cref{T3.2-5}.

We now start imposing additional assumptions on the function $f$ and the other parameters introduced above.
{Recall that $\mathcal{F}_{lip}\cup\mathcal{F}_{\ovclm+ss^*}\subset \mathcal{F}_{\ovclm}$.}
The presentation splits into considering two cases.

{\em Case 1: $f\in\mathcal{F}_{lip}$.}
Let $\lip f(\bx)<\ga<+\infty$ and $\ga\,'\in(\lip f(\bx),\ga)$.
In view of \cref{Tr1,Tr3},
we have for all sufficiently large $k\in\N$:
\begin{align}
\notag
\norm{x_{1k}^*}&\leq \norm{x_{2k}^*}+ \norm{x_{k}^*}+\|x_{1k}^*+x_{2k}^*-x_k^*\|
<\gamma\,'\norm{y_{2k}^*}+\al_{k}+\eps_k
\\&
\label{Tr8}
<\gamma\,'(1+\eps_k)+\al_{k}+\eps_k<
\gamma\,'+(\ga\,'+2)\al_k<\ga.
\end{align}
Letting $\ga\downarrow\lip f(\bx)$ and observing that the function $f\in\mathcal{F}_{lip}$ is arbitrary, inequalities \cref{Tr6,Tr8} imply
$\Rad{SR}{lip}
\ge{\rm srg}_{1}F(\bx,\by)$.
Combined with the established above inequality \cref{T3.2-1}, this proves
\cref{T3.2-4}.

{\em Case 2: {$f\in\mathcal{F}_{\ovclm+ss^*}$}.}
Without loss of generality,
for each $k\in\N$, there is a $\delta_k>0$ such that $f$ is Lipschitz continuous on $B_{\delta_k}(x_k)$ with some constant $\eta_k>0$.
Set
\begin{gather}
\label{Tr9}
\eps_k:= \min\left\{\frac{\al_kt_k}{(\al_k+1)(\eta_k+1)},\de_k,\frac12\right\},
\end{gather}
and observe that $\eps_k$ satisfies \cref{Tr2}.
Thus, we can assume that the sequences defined above satisfy conditions \cref{Tr3,Tr4,Tr5,Tr6}.
Since $f$ is \ssstar at $\bx$, we have
\begin{align}
\label{Tr10}
\lim_{k\to\infty} \frac{\ang{x_{2k}^*,x_{2k}-\bx}-\ang{y_{2k}^*,f(x_{2k})}} {\|(x_{2k}^*,y_{2k}^*)\|\|(x_{2k}-\bx,f(x_{2k}))\|}=0,
\end{align}
where the denominator is nonzero for all
%sufficiently large
$k\in\N$.
We are going to show that
\begin{align}
\label{Tr11}
\lim_{k\to\infty} \frac{\ang{x_{1k}^*,x_{1k}-\bx}-\ang{y_{1k}^*,y_{1k}-\by}} {\|(x_{1k}^*,y_{1k}^*)\|\|(x_{1k}-\bx,y_{1k}-\by)\|}=0.
\end{align}
To transform \cref{Tr10} into \cref{Tr11}, we next compare the corresponding components of the two expressions one by one.
In view of \cref{Tr3,Tr4,Tr1},
we have for all $k\in\N$:
\begin{align}
\notag
\left\vert\frac{\|(x_{1k}^*,y_{1k}^*)\|} {\|(x_{2k}^*,y_{2k}^*)\|}-1\right\vert&= \frac{\vert\|(x_{2k}^*,y_{2k}^*)\|- \|(x_{1k}^*,y_{1k}^*)\|\vert}{\|(x_{2k}^*,y_{2k}^*)\|}\le
\frac{\|(x_{2k}^*+x_{1k}^*,y_{2k}^*-y_{1k}^*)\|} {\|y_{2k}^*\|}
\\&<
\label{Tr12}
2\max\{\|x_k^*\|+\eps_k, 2\eps_k\}<4\al_k.
\end{align}
Thanks to the Lipschitz continuity of $f$ near $x_k$ and its calmness near $\bx$, and in view of \cref{Tr1,Tr2}, we have for all sufficiently large $k\in\N$:
\begin{align}
\notag
\|f(x_{2k})+y_{1k}-\by\|&\le\|f(x_{2k})-f(x_{k})\|+ \|f(x_{k})+y_{k}-\by\|+\|y_{1k}-y_{k}\|
\\&
\label{Tr13}
<\eta_k\eps_k+\al_kt_k+\eps_k
=(\eta_k+1)\eps_k+\al_kt_k,
%\le \al_kt_k\left(1+\frac{\ga+1}{\al_k+1}\right).
\\
\label{Tr14}
\|f(x_{2k})\|&\le\ga\|x_{2k}-\bx\|<\ga(t_k+\eps_k)< \ga\, t_k(1+\al_k(\al_k+1)\iv)<2\ga\, t_k,
\\
\label{Tr15}
\|x_{2k}^*\|&\le\eta_k\|y_{2k}^*\|<\eta_k(1+\eps_k) <2\eta_k.
\end{align}
As a consequence,
employing \cref{Tr5,Tr13,Tr9,Tr2},
we obtain:
\begin{align}
\notag
\left\vert\frac{\norm{(x_{1k}-\bx,y_{1k}-\by)}} {\norm{(x_{2k}-\bx,f(x_{2k}))}}-1\right\vert&= \frac{|\norm{(x_{2k}-\bx,f(x_{2k}))}- \norm{(x_{1k}-\bx,y_{1k}-\by)}|} {\norm{(x_{2k}-\bx,f(x_{2k}))}}
\\&
\notag
\le \frac{\norm{(x_{2k}-x_{1k},f(x_{2k})+y_{1k}-\by)}} {\norm{x_{2k}-\bx}}
\\&
\label{Tr16}
<\frac{(\eta_k+1)\eps_k+\al_kt_k+\eps_k} {(\al_k+1)\iv t_k}< \al_k(\al_k+3).
\end{align}
Similarly, employing \cref{Tr3,Tr4,Tr5,Tr13,Tr14,Tr15}, we obtain:
\begin{align}
\notag
\frac{|\ang{x_{2k}^*,x_{2k}-\bx}+\ang{x_{1k}^*,x_{1k}-\bx}|} {\|(x_{1k}^*,y_{1k}^*)\|\|(x_{1k}-\bx,y_{1k}-\by)\|}&\le
\frac{\|x_{1k}^*+x_{2k}^*\|\|x_{1k}-\bx\|+ \|x_{2k}^*\|\|x_{2k}-x_{1k}\|} {\|y_{1k}^*\|\|x_{1k}-\bx\|}
\\
\label{Tr17}
&<2\left(\|x_{k}^*\|+\eps_k+ \frac{4\eta_k\eps_k}{(\al_k+1)\iv t_k}\right)<12\al_k,
\\
\notag
\frac{|\ang{y_{2k}^*,f(x_{2k})}+\ang{y_{1k}^*,y_{1k}-\by}|} {\|(x_{1k}^*,y_{1k}^*)\|\|(x_{1k}-\bx,y_{1k}-\by)\|}&\le
\frac{\|y_{2k}^*-y_{1k}^*\|\|f(x_{2k})\|+ \|y_{1k}^*\|\|f(x_{2k})+y_{1k}-\by\|} {\|y_{1k}^*\|\|x_{1k}-\bx\|}
\\
\notag
&<\frac{\al_k+1}{t_k} \big(8\ga\eps_kt_k+(\eta_k+1)\eps_k+\al_kt_k\big)
\\
\notag
&=
(\al_k+1) \big(8\ga\eps_k+(\eta_k+1)\eps_k/t_k+\al_k\big)
\\
\label{Tr18}
&<
\al_k\big((\al_k+1)(8\ga+1)+1\big).
\end{align}
Conditions \cref{Tr12,Tr16,Tr17,Tr18,Tr1} yield
\begin{gather}
\label{Tr19}
\lim_{k\to\infty} \frac{\|(x_{1k}^*,y_{1k}^*)\|}{\|(x_{2k}^*,y_{2k}^*)\|} =\lim_{k\to\infty}\frac{\norm{(x_{1k}-\bx,y_{1k}-\by)}} {\norm{(x_{2k}-\bx,f(x_{2k}))}} =1,
\\
\label{Tr20}
\lim_{k\to\infty} \frac{|\ang{x_{2k}^*,x_{2k}-\bx}+\ang{x_{1k}^*,x_{1k}-\bx}|} {\|(x_{1k}^*,y_{1k}^*)\|\|(x_{1k}-\bx,y_{1k}-\by)\|}=
\lim_{k\to\infty} \frac{|\ang{y_{2k}^*,f(x_{2k})}+\ang{y_{1k}^*,y_{1k}-\by}|} {\|(x_{1k}^*,y_{1k}^*)\|\|(x_{1k}-\bx,y_{1k}-\by)\|}=0.
\end{gather}
Combining \cref{Tr19,Tr20} with \cref{Tr10}, we show \cref{Tr11}.
Letting $\ga\downarrow\clm f(\bx)$ and observing that the function {$f\in\mathcal{F}_{\ovclm+ss^*}$} is arbitrary, conditions \cref{Tr6,Tr11}, and definition \cref{srg4} imply
$\Rad{SR}{\ovclm+ss^*}
\ge{\rm srg}_{4}F(\bx,\by)$.
Combined with the established above inequality \cref{T3.2-3}, this proves
\cref{T3.2-7}.

If $f\in\mathcal{F}_{lip+ss^*}$ and $\ga>\lip f(\bx)$, then, in addition to \cref{Tr6,Tr11}, we also have condition \cref{Tr8}.
Letting $\ga\downarrow\lip f(\bx)$, conditions \cref{Tr6,Tr8,Tr11}, and definition \cref{srg3} imply \cref{T3.2-6}.

\subsection{The upper estimate for the radius of strong subregularity}

We prove \cref{L3.2} which provides
%the
a
key ingredient for the proof of \cref{T3.3}.
The proof below is an adaptation of the one in \cref{S4.3} for \cref{L3.1}\eqref{L3.1(iii)}.
As in \cref{S4.3}, we construct a function $f$ such that $F+f$ is not strongly regular, but we make this conclusion by the definition of strong regularity without appealing to \cref{T2.3,LemEpsCoder}.
Hence, we  skip all the estimates for coderivatives involved in \cref{S4.3}.
This also explains why $\gph F$ in the statement of \cref{L3.2}
is not required to be closed.
%\AK{26/02/22.
%Am I missing anything here?}

Suppose that $\srg<\ga<+\infty$.
By
%equation
\cref{ssrg}, there exist
sequences
$x_k\to\bx$ with $x_k\ne\bx$,
and $y_k\in F(x_k)$
such that condition \cref{6.2-1} is satisfied.
For every $k\in\N$,
set $t_k:=\norm{x_k-\bx}$, $u_k:=(x_k-\bx)/t_k\in\Sp_X$ and $v_k:=(y_k-\by)/t_k\in\ga\,'\B_Y$, and
choose a vector
%$u_k^*\in\Sp_{X^*}$,
%$y_k^*\in\Sp_{Y^*}$ and
$u_k^*\in X^*$ satisfying
\begin{gather}
\label{S4.5-1}
\ang{u_k^*,u_k}=\|u_k^*\|=1.
%\quad \ang{y_k^*,y_k-\by}=\|y_k-\by\| \AND x_k^*=u_k^*\|y_k-\by\|/t_k.
\end{gather}
%Then
%\[\ang{x_k^*,x_k-\bx}-\ang{y_k^*,y_k-\by}=0.\]
Set
\begin{align}
\label{S4.5-3}
\alpha_k(x):=\ang{u_k^*,x},\quad
r_k(x):=x-\alpha_k(x)u_k,\quad
A_k(x):=
\alpha_k(x)v_k.
\end{align}
Observe that
\begin{gather}
\label{S4.5-4}
\al_k(u_k)=1,
\quad
r_k(u_k)=0,
\quad
A_k(u_k)=v_k.
\end{gather}

{\em Case 1: no element of the sequence $(u_k)$ appears infinitely many times.}
As shown in \cref{S4.3}, there exists a subsequence of $(u_k)$ (we keep the original notation) such that $\rho_k:=\inf\{\|u_j-u_k\|\mid j\ne k\}>0$ for all $k\in\N$, and condition \cref{T6.2P5} is satisfied.
%{Let $\ga''\in(\ga',\ga)$.}
For each $k\in\N$, set $\tau_k:=\rho_k/2$ and
consider
mappings $s_k:X\to\R$, $\hat A_k:X\to Y$ ($k\ge1$) and $f:X\to\R$ defined by \cref{T6.2P8-1,T6.2P8-3,T6.2P9}, respectively.
Observe that $s_k(x)=0$ and $\hat A_k(x)=0$ for all $x\in X$ with $\norm{r_k(x)}\ge\tau_k\alpha_k(x)$.
If $\norm{r_k(x)}<\tau_k\alpha_k(x)$, then,
by \cref{S4.5-1,S4.5-3}, $x\ne0$ and
\begin{align*}
\big\|x-\norm{x}u_k\big\|= \norm{r_k(x)+(\alpha_k(x)-\norm{x})u_k}\le& \|r_k(x)\|+\norm{x}-\alpha_k(x)
\\\le&
2\|r_k(x)\|<2\tau_k\alpha_k(x)\le\rho_k\|x\|;
\end{align*}
hence, $x/\norm{x}\in B_{\rho_k}(u_k)$, and it follows from \cref{T6.2P5} that
for every $x\in X$ there is at most one $k\in\N$ with
$\hat A_k(x)\ne0$.
Thus, the function $f$
is well defined by \cref{T6.2P9}.
The function $\hat A_k(x)$ is Lipschitz continuous around any point $x\ne0$; hence, $f$ is Lipschitz continuous around any point $x\ne\bx$.
Further, $s_k$ is positively homogenous at $0$, while
$A_k$ is
%continuous
linear.
Therefore $\hat A_k$ is positively homogenous at $0$, and $f$
is positively homogenous at $\bx$ and, by \cref{CorSS}, \ssstar at $\bx$.
If $\hat A_k(x)\ne0$, then, by \cref{T6.2P8-1,T6.2P8-3,S4.5-1,S4.5-3,6.2-1}, $\norm{r_k(x)}<\tau_k\alpha_k(x)$ and
\begin{align}
\label{S4.5-7}
\|\hat A_k(x)\|\le&\|A_k(x)\|=\alpha_k(x)\|v_k\|\leq
\ga'\|x\|,
\end{align}
and consequently,
$\clm\hat A_k(0)\leq\ga'$.
Hence, $\clm f(\bx)\leq\ga'<\ga$.
Thus, $f\in {\mathcal F}_{{\ovclm}+ss^*}$.

In view of \cref{T6.2P8-1,S4.5-4},
$s_k(x_k-\bx)=1$.
This yields \cref{T6.2P15}.
By \cref{T6.2P15},
%and \cref{T6.2P17},
$\by\in(F+f)(x_k)$, and in view of
%the definition
\cref{ssrg}, srg$(F+f)(\bx,\by)=0$, i.e., $F+f$ is not strongly subregular at $(\bx,\by)$.
Hence, \cref{Eq_ssrgLB} holds true.

{\em Case 2: There is an element $u$ of the sequence $(u_k)$ which appears infinitely many times.}
From now on, we consider the stationary subsequence $(u)$ of the sequence $(u_k)$, i.e. without changing the notation, we assume that $(x_k-\bx)/t_k=u\in\Sp_{X}$ for all $k\in\N$.
Thus, subscript $k$ can be dropped in many of the formulas in \cref{S4.5-1,S4.5-3,S4.5-4}:
\begin{gather}
\label{S4.5-8}
\ang{u^*,u}
=\|u^*\|
=1,\;\;
\alpha(x):=\ang{u^*,x},\;\;
r(x):=x-\alpha(x)u,
\;\;
A_k(x):=\al(x)v_k,
\\
\label{S4.5-9}
\al(u)=1,
\quad
r(u)=0.
\end{gather}
By passing to a subsequence again, we can also assume that the sequence $(t_k)$ satisfies conditions \cref{T6.2P22}.
%Choose a number $\tau\in(0,1)$ and set
Set
\begin{align}
\label{S4.5-10}
s(x):=&
\begin{cases}
\max\left\{1-\left(\dfrac{\norm{r(x)}} {\alpha(x)}\right)^2,0\right\}&\mbox{if $\alpha(x)>0$},\\
   0&\mbox{otherwise},
\end{cases}
\\
\label{S4.5-11}
\hat A_k(x):=&s(x)A_k(x).
\end{align}
Next we define mappings $T_k:\{x\in X\mid \alpha(x)>0\}\to Y$ ($k\ge2$) and $g,f:X\to Y$ by \cref{T6.2P23.2,T6.2P31,T6.2P32}.
\if{
The mappings $T_k$ are Lipschitz continuous around every
%$x\not=0$
$x\in X$ with $\al(x)>0$
as compositions of mappings having this property, and satisfy
\cref{T6.2P24}.
Hence, $g$ is Lipschitz continuous around every point
$x\in U:=\{x\in X\mid 0<\norm{x}< t_1\}$ with $\al(x)>0$.
The function $s$ vanishes on the open set $U_0:=\{x\mid\alpha(x)<\|r(x)\|\}$.
By \cref{T6.2P31,T6.2P23.2,S4.5-11},
so does $g$.
%Hence, $g$ is Lipschitz continuous around every point in $U_0$.
If $\alpha(x)=\|r(x)\|=0$, then, by \cref{S4.5-8}, $x=0$.
Hence, $\al(x)>0$ for all $x\in U\setminus U_0$, and consequently, $g$ is Lipschitz continuous around every point in $U$.
The mapping $\hat A_k$ is positively homogeneous at $0$, and
satisfies \cref{S4.5-7}.
As a consequence,
$\|g(x)\|\le\ga\,'\|x\|$ for all $x\in U\setminus U_0$,
showing $\clm g(0)\leq \ga'$.
As shown in \cref{S4.3}, for all $x\in U\setminus U_0$, it holds
\begin{gather*}
%\label{T6.2P25}
\max\{\|g(x)-g'(x;x)\|,\|g(x)+g'(x;-x)\|\}<
2\ga\|x\|/\hat k(x),
\end{gather*}
where $\hat k(x)\to\infty$ as $x\to0$.
It follows from \cref{LemSS} that $g$ is \ssstar at $0$.
We conclude that the function $f$ defined by \cref{T6.2P32} belongs to
${\mathcal F}_{{\ovclm}+ss^*}$, and
$\clm f(\bx)<\ga$.
}\fi
As shown in \cref{S4.3}, ${f\in\mathcal F}_{{\ovclm}+ss^*}$, and
$\clm f(\bx)<\ga$.
By \cref{T6.2P23.2,T6.2P31,T6.2P32,S4.5-8,S4.5-9,S4.5-10,S4.5-11}, $\al(t_ku)=t_k$, and
\begin{align*}
f(x_k)=
%f(\bx+t_ku)=
-g(t_ku)=-T_k(t_ku)&=-\hat A_k(t_ku)=
-s(t_ku)A_k(t_ku)
%\\&
%=-A_k(t_ku)
=-t_kv_k=\by-y_k.
\end{align*}
Thus, $\by\in(F+f)(x_k)$, and in view of
%the definition
\cref{ssrg}, srg$(F+f)(\bx,\by)=0$, i.e., $F+f$ is not strongly subregular at $(\bx,\by)$.
Hence, \cref{Eq_ssrgLB} holds true.

%\section*{Acknowledgement}

%The authors wish to thank the referees
\section*{Declarations}

\noindent{\bf Funding. }The second author benefited
from the support of the Australian Research Council, project DP160100854, and the European Union's Horizon 2020 research and innovation
programme under the Marie Sk{\l }odowska--Curie Grant Agreement No. 823731
CONMECH.

\noindent{\bf Conflict of interest.} The authors have no competing interests to declare that are relevant to the content of this article.

\noindent{\bf Data availability. }
Data sharing is not applicable to this article as no datasets have been generated or analysed during the current study.

%\addcontentsline{toc}{section}{References}

\bibliographystyle{spmpsci}
\bibliography{buch-kr,kruger,kr-tmp}

\end{document}